\documentclass[12pt,twoside]{amsart}
\usepackage{latexsym,amsfonts,amsmath,amssymb,soul}
\usepackage{enumerate,soul}
\usepackage[usenames, dvipsnames]{color}
\numberwithin{equation}{section}

\makeatletter
\@namedef{subjclassname@2020}{%
  \textup{2020} Mathematics Subject Classification}
\makeatother

\newcommand{\norm}[1]{\left\lVert#1\right\rVert}

\newcommand\restr[2]{{
  \left.\kern-\nulldelimiterspace 
  #1 
  \right|_{#2} 
}}
\newcommand{\R}{\mathbb{R}}

\DeclareMathOperator*{\ces}{ces}
\DeclareMathOperator*{\proj}{proj}

\DeclareMathOperator*{\esup}{ess \ sup}
\DeclareMathOperator*{\supp}{supp}
\DeclareMathOperator*{\Wig}{Wig}

\DeclareMathOperator*{\op}{op}

\AtBeginDocument{
\def\MR#1{}}

\let\OLDthebibliography=\thebibliography
\def\thebibliography#1{
\OLDthebibliography{#1}
\addcontentsline{toc}{section}{\refname}}

\newtheorem{theo}{Theorem}[section]
\newtheorem{prop}[theo]{Proposition}
\newtheorem{defin}[theo]{Definition}
\newtheorem{lema}[theo]{Lemma}
\newtheorem{cor}[theo]{Corollary}

\theoremstyle{remark}
\newtheorem{nota}[theo]{Remark}
\newtheorem{exam}[theo]{Example}
\title{On the compactness of the Weyl operator in $\mathcal{S}_{\omega}$}

\author[Asensio]{Vicente Asensio}
\address{Instituto Universitario de Matem\'atica Pura y Aplicada IUMPA\\
Universitat Po\-li\-t\`ecni\-ca de Val\`encia\\
Camino de Vera, s/n\\
E-46022 Val\`encia\\
Spain}
\address{Centro Universitario EDEM \\
Pla\c{c}a de l'Aigua, s/n\\
E-46024 Val\`encia\\
Spain
}
\email{viaslo@upv.es, vasensio@edem.es}
\author[Boiti]{Chiara Boiti}
\address{
Dipartimento di Matematica e Informatica \\Universit\`a di Ferrara\\
Via Ma\-chia\-vel\-li n.~30\\
I-44121 Ferrara\\
Italy}
\email{chiara.boiti@unife.it}

\author[Jornet]{David Jornet}
\address{
Instituto Universitario de Matem\'atica Pura y Aplicada IUMPA\\
Universitat Po\-li\-t\`ecni\-ca de Val\`encia\\
Camino de Vera, s/n\\
E-46022 Val\`encia\\
Spain}
\email{djornet@mat.upv.es}

\author[Oliaro]{Alessandro Oliaro}
\address{Dipartimento di Matematica\\ Universit\`a di Torino\\
 Via Carlo Alberto n.~10\\ I-10123 Torino\\ Italy}
 \email{alessandro.oliaro@unito.it}

\begin{document}

\maketitle

\begin{abstract}
	We characterize, using time-frequency analysis, the continuity and compactness of the Weyl operator in global classes of ultradifferentiable functions $\mathcal{S}_\omega$, for weight functions $\omega$ in the sense of Braun, Meise and Taylor. As a consequence, we give results about the compactness of the localization operator in $\mathcal{S}_\omega$, in relation with the spaces of $\omega$-multipliers and $\omega$-convolutors of $\mathcal{S}_\omega$. Moreover, we provide several examples that complement our investigation. 
\end{abstract}

\section{Introduction}

In 1988, Daubechies~\cite{Daubechies} introduced localization operators to localize a signal both in time and frequency, although a particular case had already been considered by Berezin in 1971. They are also known in some cases as Toeplitz operators or anti-Wick operators. These operators were traditionally defined on $L^{2}$ or on some Sobolev spaces, but they can be defined on more general spaces, called modulation spaces. The modulation spaces $M^{p,q}$ $(1\le p,q\le \infty)$ were introduced by Feichtinger as appropriate classes of functions or distributions to discuss certain problems related to time-frequency analysis. For $p = q = 2$, we get $M^{p,q} = L^2$ and the appropriate weighted versions of the modulation spaces produce Sobolev spaces. Several authors have studied the compactness of localization operators when acting on modulation spaces.   In \cite{FG-Compact}, the authors characterized the symbols $F$ in the modulation class $M^\infty(\R^{2d})$ with the property that the localization operator with symbol $F$ and windows $\varphi,\psi$, denoted by $L^F_{\varphi,\psi}$, is a compact operator when acting on $L^2(\R^d)$ for every pair of windows $\varphi,\psi$ in the Schwartz class. Boggiatto~\cite{boggiatto} proved that each $F\in L^{\infty}(\R^{2d})$ vanishing at infinity defines a compact localization operator on $M^{p,q}(\R^d)$. However, there exists a compact localization operator whose symbol $F$ is a bounded function with constant modulus, thus not vanishing at infinity~\cite{FG-Some}. Boundedness and Schatten-class conditions of localization operators were  investigated in \cite{CG,CG1}. Boiti and De Martino \cite{BdM} gave a sufficient condition for the compactness of the localization operator in modulation spaces with exponential weights, thus extending partially the work in \cite{FG-Compact}.
Bastianoni, Cordero and Nicola studied in \cite{BCN} properties of eigenfunctions of compact localization operators.

  One of earliest forms of pseudodifferential operator arises in quantum mechanics in 1931 introduced by Weyl as a rule that associates a self-adjoint operator $a^w(x,D)$ with a real-valued symbol $a(x,\xi)$ defined in $\R^{2d}$. Localization operators can be also represented as this type of pseudodifferential operators, the so-called Weyl operators. Sj\"ostrand~\cite{Sj} showed that if the symbol belongs to the modulation class $M^{\infty,1}(\R^{2d})$ (called Sj\"ostrand class), then the corresponding Weyl operator is continuous on $L^{2}(\R^{d})$. Continuity of Weyl operators on modulation classes was studied by Gr\"ochenig, Heil, Toft and Czaja, 
  among others~\cite{Cz,GH, T}. A very important tool is the study of the continuity of the Wigner transform between different classes of modulation spaces~\cite{T}. In particular, if the symbol belongs to the Sj\"ostrand class  and $1 \le p,q \le \infty$, then the corresponding Weyl operator is bounded on $M^{p,q}(\R^d)$ \cite[14.5.2]{G-Found} and, if the symbol can be approximated in $M^{\infty,1}(\R^{2d})$ by functions in the Schwartz class, then the corresponding Weyl operator is compact on $M^{p,q}(\R^d)$  (see, for instance, \cite{BGHO}).  In \cite{FG-Some}, the authors proved that, if the symbol is a tempered distribution and, for some  $ 1 < p, q < \infty$, the  corresponding Weyl operator is compact on $M^{p,q}(\R^d)$, then the short-time Fourier transform (STFT) of the symbol vanishes at
infinity, that is, it belongs to $M^{\infty}(\R^{2d})$ and it is the limit in $M^{\infty}(\R^{2d})$ of a sequence of functions in the Schwartz class. Moreover, when the symbol belongs to the Sj\"ostrand class, the authors characterized when the Weyl operator is compact for every $1\le p,q\le \infty$. 
  
  In this work, we characterize the continuity and compactness of the  Weyl operator when defined in $\mathcal{S}_\omega(\R^{d})$, modulated by a weight function $\omega$ in the sense of Braun, Meise and Taylor~\cite{BMT}. We use purely techniques from time-frequency analysis and give the characterization in terms of the STFT of the symbol of the operator. In our setting, the importance of the class $\mathcal{S}_\omega(\R^{d})$  is twofold. On the one hand, we obtain a scale of classes including, for particular selections of $\omega$, the Schwartz class, or, for instance, many Gelfand-Shilov spaces. On the other hand, these classes appear in a natural way in time-frequency analysis in relation with modulation spaces with exponential weights, since they are the intersection of modulation spaces with exponential weights. Therefore, several problems related with these classes can be studied with time-frequency analysis. We also give some consequences to the compactness of localization operators on $\mathcal{S}_\omega(\R^{d})$ and obtain relations  with the spaces of $\omega$-multipliers and $\omega$-convolutors in the sense of \cite{AM-Mult,AM}.

  The paper is organized as follows. In Section~\ref{preli} we give the necessary notation, definitions and known results used in the following sections.  In Section~\ref{mult-convo} we characterize the spaces of multipliers and convolutors introduced in \cite{AM-Mult,AM} in terms of the STFT, thus extending the previous work in \cite{BO}. Since a particular case of the Weyl operator is the multiplication operator, we study in Section~\ref{SectCompMult} the compactness of the multiplication operator on $\mathcal{S}_\omega(\R^{d})$. In Section~\ref{Weyl oper} we characterize the compactness of the Weyl operator on $\mathcal{S}_\omega(\R^{d})$ and provide several examples. Finally, in Section~\ref{SectLoc}, we give consequences for the study of the compactness of the localization operator in  $\mathcal{S}_\omega(\R^{d})$ and give a concrete example of a compact localization operator whose symbol is not an $\omega$-convolutor. We also obtain some further results about the compactness of the localization operator in related spaces. 

\section{Preliminaries}\label{preli}
We consider weight functions in the sense of Braun, Meise, and Taylor~\cite{BMT}:
\begin{defin}
A {\em non-quasianalytic weight function} $\omega:[0,+\infty[ \to [0,+\infty[$ is an increasing and continuous function satisfying:
\begin{itemize}
\item[$(\alpha)$] There exists $L\geq1$ such that $\omega(2t) \leq L\omega(t) + L$, $t\geq0$;
\item[$(\beta)$] $\int_1^{+\infty} \frac{\omega(t)}{1+t^2} dt < +\infty$;
\item[$(\gamma)$] $\log(t) = O(\omega(t))$ as $t \to \infty$; 
\item[$(\delta)$] The function $\varphi_{\omega}(t) = \omega(e^t)$ is convex.
\end{itemize}
\end{defin}
A weight function, or simply, weight, is extended to $\mathbb{C}^d$ in the following way: $\omega(\xi) = \omega(|\xi|)$, $\xi \in \mathbb{C}^d$, where $|\xi|$ denotes the Euclidean norm in $\mathbb{C}^d$. 
The \emph{Young conjugate $\varphi^{\ast}_{\omega}:[0,+\infty[ \to [0,+\infty[$ of $\varphi_{\omega}$} is given by $\varphi^{\ast}_{\omega}(t) = \sup_{s\geq0} \{ st-\varphi_{\omega}(s)\}$.

In the following we will consider also subadditive weights, i.e. weights satisfying
$$ (\alpha') \qquad \qquad \omega(s+t) \leq \omega(s) + \omega(t), \quad s,t \geq 0. $$
We observe that $(\alpha')$ is stronger than $(\alpha)$.

The \emph{Fourier transform} of $f \in L^1(\mathbb{R}^d)$ is denoted by
$$ \widehat{f}(\xi) = \mathcal{F}(f)(\xi) := \int_{\mathbb{R}^d} e^{-it\cdot \xi} f(t) dt, \qquad \xi \in \mathbb{R}^d. $$
Then, the inverse of the Fourier transform is given by
$$ \mathcal{F}^{-1}(f)(x) = (2\pi)^{-d} \int_{\mathbb{R}^d} e^{i\xi \cdot x} f(\xi) d\xi, \qquad x \in \mathbb{R}^d. $$
When $f \in L^1(\mathbb{R}^{2d})$, we can also consider \emph{partial Fourier transforms}. For example, the second partial Fourier transform is denoted by
$$ \mathcal{F}_2(f)(x,\xi) := \mathcal{F}_{t \mapsto \xi}(f)(x,t) = \int_{\mathbb{R}^d} e^{-it\cdot \xi} f(x,t) dt, \qquad (x,\xi) \in \mathbb{R}^{2d}, $$
and the inverse of the partial Fourier transform is given analogously. The definitions of Fourier and partial Fourier transform can be extended in a standard way to more general spaces of functions or distributions.

\begin{defin}
For a weight function $\omega$, we define the space $\mathcal{S}_{\omega}(\mathbb{R}^d)$ as those $f \in L^1(\mathbb{R}^d)$ such that ($f, \widehat{f} \in C^{\infty}(\mathbb{R}^d)$ and) for all $\lambda>0$ and $\alpha \in \mathbb{N}_0^d$,
$$ \sup_{x \in \mathbb{R}^d} |D^{\alpha} f(x)| e^{\lambda\omega(x)} < +\infty \qquad \text{and} \qquad \sup_{\xi \in \mathbb{R}^d} |D^{\alpha} \widehat{f}(\xi)| e^{\lambda\omega(\xi)} < +\infty. $$
We denote by $\mathcal{S}'_{\omega}(\mathbb{R}^d)$ the strong dual space of $\mathcal{S}_{\omega}(\mathbb{R}^d)$.
\end{defin}
\noindent We refer to \cite{BJO-Regularity, RealPW} for equivalent characterizations of 
$\mathcal{S}_{\omega}(\mathbb{R}^d)$.
 If $\omega(t)=\log(1+t)$, it is not difficult to see that  $\mathcal{S}_{\omega}(\mathbb{R}^d)=\mathcal{S}(\mathbb{R}^d)$, i.e. the Schwartz class.

Now, we recall some basic notation from time-frequency analysis that we will use in the paper. The \emph{translation, modulation}, and \emph{phase-shift operators} are denoted, for $z=(x,\xi)\in\mathbb{R}^{2d}$, by
$$ T_xf(t) = f(t-x), \qquad \ M_{\xi}f(t) = e^{it \cdot \xi} f(t), \qquad \ \Pi(z)f(t) = M_{\xi} T_x f(t) = e^{it\cdot \xi} f(t-x). $$
Given  $\psi \in \mathcal{S}_{\omega}(\mathbb{R}^d) \setminus \{0\}$, that we will call \emph{window}, the \emph{short-time Fourier transform} of $f \in \mathcal{S}'_{\omega}(\mathbb{R}^d)$ (or simply STFT) is denoted, for $z=(x,\xi) \in \mathbb{R}^{2d}$, by
$$ V_{\psi}f(z) := \langle f, \Pi(z)\psi \rangle = \int_{\mathbb{R}^d} f(y) \overline{\psi(y-x)} e^{-iy\cdot \xi} dy.$$
Here the bracket $\langle \cdot, \cdot \rangle$ is consistent with the inner product in $L^2$, since we consider distributions as conjugate-linear functionals. 

By \cite[Theorem 2.7]{GZ}, given any $\psi \in \mathcal{S}_{\omega}(\mathbb{R}^d) \setminus \{0\}$, we have that $f \in \mathcal{S}'_{\omega}(\mathbb{R}^d)$ belongs to $\mathcal{S}_{\omega}(\mathbb{R}^d)$ if and only if, for any $\lambda>0$,
$$ p_\lambda(f):=\sup_{z \in \mathbb{R}^{2d}} |V_{\psi}f(z)| e^{\lambda\omega(z)} < +\infty.$$
Moreover, the family of seminorms $\{p_\lambda \}_{\lambda>0}$ defines a fundamental system of seminorms for $\mathcal{S}_{\omega}(\mathbb{R}^d)$ \cite[Proposition 2.10]{Gabor} (see \cite[Corollary 11.2.6]{G-Found} for polynomial weights). See also~\cite[Theorem 9]{Asensio-Wigner}. If $\mathcal{I}f(t) = f(-t)$ denotes the reflection operator of any function $f\in L^2(\R^d)$, the \emph{cross-Wigner transform $\Wig(g,f)$}  of two functions $f,g\in L^2(\R^d)$ is denoted, for $z=(x,\xi) \in \mathbb{R}^{2d}$, by
\begin{equation}\label{EqWig}
\Wig(g,f)(z) := 2^d e^{2ix\cdot \xi} \langle g, \Pi(2z) \mathcal{I}f \rangle = \int_{\mathbb{R}^d} g(x+y/2) \overline{f(x-y/2)} e^{-iy\cdot \xi} dy.
\end{equation}
If $g,f \in \mathcal{S}_{\omega}(\mathbb{R}^d)$, then we have $\Wig(g,f) \in \mathcal{S}_{\omega}(\mathbb{R}^{2d})$ since $\mathcal{S}_{\omega}(\mathbb{R}^{2d})$ is invariant under partial Fourier transform. We write $\Wig f := \Wig(f,f)$. 

We also 
introduce, following~\cite{BuzanoO} (cf.~\cite{BJO-Regularity}), the \emph{Wigner-like transform of $f \in \mathcal{S}_{\omega}(\mathbb{R}^{2d})$} by
\begin{equation}\label{Wigner-like}
\Wig[f](z) := \mathcal{F}_2(\mathcal{T}f)(z) = \int_{\mathbb{R}^d} f(x+y/2,x-y/2) e^{-iy\cdot \xi} dy, \qquad z=(x,\xi)\in\mathbb{R}^{2d}, 
\end{equation} 
where
\begin{equation}\label{EqChangeT}
\mathcal{T}f(x,y) = f(x+y/2,x-y/2) \qquad \qquad \mathcal{T}^{-1}f(x,y) = f\Big(\frac{x+y}{2}, x-y\Big).
\end{equation}
We observe that $\Wig[f]$ is invertible in $\mathcal{S}_{\omega}(\mathbb{R}^{2d})$ and in $\mathcal{S}'_{\omega}(\mathbb{R}^{2d})$ since $\mathcal{T}$ is a linear invertible change of variables, and the inverse is given by $\Wig^{-1}[f] = \mathcal{T}^{-1}\mathcal{F}_2^{-1}(f)$.

For $a \in \mathcal{S}'_{\omega}(\mathbb{R}^{2d})$, the \emph{Weyl operator} $a^w(x,D)$ applied to $f \in \mathcal{S}_{\omega}(\mathbb{R}^d)$ is the distribution defined by 
\begin{equation}\label{Weyl-oper}
	\langle a^w(x,D)f, g \rangle = (2\pi)^{-d} \langle a, \Wig(g,f) \rangle, \qquad g \in \mathcal{S}_{\omega}(\mathbb{R}^d).
\end{equation}  
So, we have that $a^w(x,D): \mathcal{S}_{\omega}(\mathbb{R}^d) \to \mathcal{S}'_{\omega}(\mathbb{R}^d)$. The Weyl operator can be interpreted as
\begin{equation}\label{EqWeylOperator}
a^w(x,D)f (x)= (2\pi)^{-d} \int_{\mathbb{R}^{2d}} e^{i(x-s)\cdot \xi} a\Big(\frac{x+s}{2},\xi\Big) f(s) ds d\xi, \qquad f \in \mathcal{S}_{\omega}(\mathbb{R}^d).
\end{equation}
By \cite[Lemma 3.3]{Asensio}  and \cite[Theorem 3.7]{AJ}, if $a(x,\xi)$ is a function satisfying~\cite[Definition 3.1]{AJ}, then the operator~\eqref{EqWeylOperator} is continuous from $\mathcal{S}_{\omega}(\mathbb{R}^d)$ into itself; see also~\cite[Section 3]{Asensio}.

The \emph{localization operator} with symbol $a\in \mathcal{S}'_{\omega}(\mathbb{R}^{2d})$ and windows $\psi,\gamma \in \mathcal{S}_{\omega}(\mathbb{R}^d) \setminus \{0\}$ is defined as the composition of the following three operators: 
\begin{equation}\label{EqLocMult}
L^a_{\psi, \gamma} = V^{\ast}_{\gamma} \circ {\mathcal M}_a \circ V_{\psi},
\end{equation}
where $V_{\psi}:\mathcal{S}_{\omega}(\mathbb{R}^d)\to \mathcal{S}_{\omega}(\mathbb{R}^{2d}) $ is the short-time Fourier transform operator with respect to the window  $\psi$, ${\mathcal M}_a: \mathcal{S}_{\omega}(\mathbb{R}^{2d})\to \mathcal{S}_{\omega}'(\mathbb{R}^{2d})$ is the \emph{multiplication operator} by the ultradistribution $a\in \mathcal{S}_{\omega}'(\mathbb{R}^{2d})$, i.e. ${\mathcal M}_a(f):=a f,$ for $f\in \mathcal{S}_{\omega}(\mathbb{R}^{2d})$, and $V^{\ast}_{\gamma}:\mathcal{S}_{\omega}'(\mathbb{R}^{2d})\to \mathcal{S}_{\omega}'(\mathbb{R}^{d})$ is the {adjoint  of the short-time Fourier transform} with respect to the window $\gamma$. It is obvious that, if $a \in \mathcal{S}'_{\omega}(\mathbb{R}^{2d})$ and $\psi,\gamma \in \mathcal{S}_{\omega}(\mathbb{R}^d) \setminus \{0\}$, then, by the continuity of the STFT in $\mathcal{S}_{\omega}(\mathbb{R}^d)$~\cite[Prop. 2.9]{Gabor}, the localization operator $L^a_{\psi,\gamma}: \mathcal{S}_{\omega}(\mathbb{R}^d) \to \mathcal{S}'_{\omega}(\mathbb{R}^d)$ is linear and continuous (see also, for example,~\cite[Section 1]{CG}, and the references therein). 
For $a \in \mathcal{S}'_{\omega}(\mathbb{R}^{2d})$, the localization operator can be seen as a Weyl operator via the formula (see, for instance~\cite{CG}, and the references therein)
\begin{equation}\label{EqRelLocWeyl}
L^a_{\psi, \gamma} = (a \ast \Wig(\gamma, \psi))^w(x,D),
\end{equation}
that is, the \emph{Weyl symbol} of $L^a_{\psi,\gamma}$ is given by
$ a \ast \Wig(\gamma, \psi). $

If $E$ and $F$ are two Hausdorff locally convex spaces and $T:E\to F$ is a linear operator, we say that $T$ is \emph{compact} (\emph{bounded}) if there exists a $0$-neighbourhood $U$ in $E$ such that $T(U)$ is relatively compact (bounded) in $F$. Observe that any compact or bounded linear operator $T:E\to F$ is also continuous. When $E=F=\mathcal{S}_{\omega}(\mathbb{R}^d)$, since this space is Fr\'echet-Montel, any bounded linear operator $T:\mathcal{S}_{\omega}(\mathbb{R}^d)\to \mathcal{S}_{\omega}(\mathbb{R}^d)$ is also compact. Our aim is to characterize the compactness of the Weyl operator $a^w(x,D):\mathcal{S}_{\omega}(\mathbb{R}^d)\to \mathcal{S}_{\omega}(\mathbb{R}^d)$, which is equivalent to characterize the boundedness of such operator. 


In the following of this section, we consider a subadditive weight function $\omega$ and the spaces $L^{p,q}_\lambda(\mathbb{R}^{2d})$, for $1 \leq p,q < +\infty$ and $\lambda\in\mathbb{R}$, of measurable functions $F:\mathbb{R}^{2d} \to \mathbb{C}$ such that
$$ \Big( \int_{\mathbb{R}^d} \Big( \int_{\mathbb{R}^d} |F(x,\xi)|^p e^{\lambda p \omega(x,\xi)} dx \Big)^{q/p} d\xi \Big)^{1/q}<+\infty, $$
 with standard extensions when $p=+\infty$ or $q=+\infty$. 
 For $1 \leq p \leq +\infty$, we write $L^p_{\lambda}(\mathbb{R}^{2d})$ for $L^{p,p}_{\lambda}(\mathbb{R}^{2d})$. For instance:
$$ L^{\infty}_\lambda(\mathbb{R}^{2d}) = \Big\{ F \ \text{measurable on} \ \mathbb{R}^{2d} \ : \ \esup_{(x,\xi) \in \mathbb{R}^{2d}} |F(x,\xi)| e^{\lambda \omega(x,\xi)} < +\infty \Big\}. $$ 
Then we denote the corresponding modulation spaces by
$$ M^{p,q}_{\lambda}(\mathbb{R}^d) := \big\{ f \in \mathcal{S}'_{\omega}(\mathbb{R}^d) \ : \ V_{\psi}f \in L^{p,q}_{\lambda}(\mathbb{R}^{2d}) \big\}, $$
endowed with the norm $\norm{f}_{M^{p,q}_{\lambda}(\mathbb{R}^d)} = \norm{V_{\psi}f}_{L^{p,q}_{\lambda}(\mathbb{R}^{2d})}$. 
The modulation spaces do not depend on the choice of the windows $\psi \in \mathcal{S}_{\omega}(\mathbb{R}^d)\setminus\{0\}$, since different windows give equivalent norms~\cite[(3.13)]{Gabor}. 
We refer to~\cite[Section 3]{Gabor} for more information on modulation spaces in this context. 

The following lemma is crucial to characterize the boundedness of the Weyl operator in this setting. The proof is straightforward and can be found in \cite[Lemma 25]{ABR} (see also \cite[Remark 4.3]{Mele}).
\begin{lema}\label{LemmaABR}
Let $X:= \proj_{m} X_m$ and $Y:= \proj_{m} Y_m$ be Fr\'echet spaces such that $X = \cap_{m \in \mathbb{N}} X_m$ with each $(X_m, \norm{\cdot}_{m})$ a Banach space and $Y = \cap_{m \in \mathbb{N}} Y_m$ with each $(Y_m, \norm{\cdot}_{m})$ a Banach space. Moreover,  assume that $X$ is dense in $X_m$ and $X_{m+1} \subseteq X_m$ with continuous inclusion for each $m \in \mathbb{N}$, and $Y_{m+1} \subseteq Y_m$ with continuous inclusion for each $m \in \mathbb{N}$. Let $T:X\to Y$ be a linear operator. We have:
\begin{itemize}
\item[(i)] $T:X\to Y$ is continuous if and only if for each $n \in \mathbb{N}$ there exists $m \in \mathbb{N}$ such that $T$ has a unique continuous linear extension $T_{m,n} : X_m \to Y_n$.
\item[(ii)]  $T:X\to Y$ is  bounded  if and only if there exists $k_0 \in \mathbb{N}$ such that, for every $m \in \mathbb{N}$, $T$ has a unique continuous linear extension $T_{k_0, m}: X_{k_0} \to Y_m$.
\end{itemize}
\end{lema}



The modulation spaces $M^{p,q}_{\lambda}(\mathbb{R}^d)$ are  Banach spaces ($\lambda \in \mathbb{R}$, $1 \leq p,q \leq +\infty$) (see \cite[Theorem 11.3.5]{G-Found} for the case of polynomial weights; the proof is similar in the present context). We also know that (see \cite[Remark 3.6]{Gabor} and \cite[Theorem $2.5(h)'$]{RealPW})
$$ \mathcal{S}_{\omega}(\mathbb{R}^d) = \bigcap_{\lambda>0} M^{\infty}_{\lambda}(\mathbb{R}^d)=\bigcap_{\lambda>0} M^{p,q}_{\lambda}(\mathbb{R}^d).$$
We have that $\mathcal{S}_{\omega}(\mathbb{R}^d)$ is a dense subspace of $M^{p,q}_{\lambda}(\mathbb{R}^d)$ for $1 \leq p,q<+\infty$ and $\lambda \in\mathbb{R}$~\cite[Proposition 3.9]{Gabor}.  It is easy to check that the inclusion $M^{p,q}_{\lambda}(\mathbb{R}^d) \hookrightarrow M^{p,q}_{\mu}(\mathbb{R}^d)$ is continuous for all $\lambda>\mu$ and $1 \leq p,q \leq +\infty$. Hence, the last intersections for $\lambda>0$ can be replaced by the intersections for $m\in \mathbb{N}$.
Since we can describe the space $\mathcal{S}_{\omega}(\mathbb{R}^d)$ with fundamental systems of seminorms with $L^{p,q}$-norms of the STFT (\cite[Proposition 2.10]{Gabor} and \cite[Theorem 2.5]{RealPW}) we have 
$$ \mathcal{S}_{\omega}(\mathbb{R}^d) = \proj_{m \in \mathbb{N}} M^{p,q}_{m}(\mathbb{R}^d). $$
Now we easily deduce the following result from Lemma~\ref{LemmaABR}:

\begin{prop}\label{LemmaKEY}
Let $\omega$ be subadditive and $T:\mathcal{S}_{\omega}(\mathbb{R}^d) \to \mathcal{S}_{\omega}(\mathbb{R}^d)$ be a linear operator, and let $1 \leq p,q < +\infty$ and $1 \leq \bar{p}, \bar{q} \leq +\infty$.
\begin{enumerate}
\item[(i)] $T$ is continuous if and only if for every $n \in \mathbb{N}$ there is $m \in \mathbb{N}$ such that the operator has a unique continuous linear extension $T_{n,m} : M^{p,q}_{m}(\mathbb{R}^d) \to M^{\bar{p}, \bar{q}}_{n}(\mathbb{R}^d)$.

\item[(ii)] $T$ is compact if and only if there exists $k_0 \in \mathbb{N}$ such that for all $m \in \mathbb{N}$ the operator has a unique  continuous linear extension $T_{k_0, m} : M^{p,q}_{k_0}(\mathbb{R}^d) \to M^{\bar{p},\bar{q}}_{m}(\mathbb{R}^d)$.
\end{enumerate}
\end{prop}
If $1\leq p,q < +\infty$, $\lambda \in \mathbb{R}$, then the dual space of $M^{p,q}_{\lambda}(\mathbb{R}^d)$  is $M^{p,q}_{\lambda}(\mathbb{R}^d)'= M^{p',q'}_{-\lambda}(\mathbb{R}^d)$~\cite[Theorem 3.8]{Gabor}, 
where $1/p + 1/p' = 1$ and $1/q + 1/q' = 1$. Moreover, 
given $\psi\in\mathcal{S}_{\omega}(\mathbb{R}^d)\setminus\{0\}$,
the duality is given by
$$ \langle f,h \rangle = \int_{\mathbb{R}^{2d}} V_{\psi} f(z) \overline{V_{\psi}h(z)} dz, $$
for $f \in M^{p,q}_{\lambda}(\mathbb{R}^d)$ and $h \in M^{p',q'}_{-\lambda}(\mathbb{R}^d)$.

\begin{exam}\label{ExamLoc1}
When $a \equiv 1$, then (see for example~\cite[(2.22)]{Gabor}) $L^{a}_{\psi,\gamma} = (2\pi)^d \langle \gamma, \psi \rangle I$, where $I$ is the identity operator on $\mathcal{S}_\omega(\mathbb{R}^d)$, which is continuous. Hence, if $\langle \gamma,\psi \rangle \neq 0$, $L^{a}_{\psi,\gamma}$ cannot be compact.
\end{exam}

\section{Multipliers and convolutors}\label{mult-convo}
We recall the following definition from \cite[Definition 3.1]{AM-Mult}:

\begin{defin}
A function $a\in C^\infty(\R^d)$ is \emph{$\omega$-slowly increasing}, or simply, an \emph{$\omega$-multiplier}   if and only if for all $\lambda>0$ there exist $C_{\lambda}, \mu_{\lambda}>0$ such that
\begin{equation}\label{EqMultDer}
|D^{\alpha} a(x)| \leq C_{\lambda} e^{\lambda \varphi^{\ast}_{\omega}\big(\frac{|\alpha|}{\lambda}\big)} e^{\mu_{\lambda}\omega(x)}, \qquad \alpha \in \mathbb{N}_0^d, \ x \in \mathbb{R}^d.
\end{equation}
	We denote $O_M^{\omega}(\mathbb{R}^d)$ for the space of all {$\omega$-multipliers}. 
\end{defin}
In \cite[Theorem 4.4]{AM-Mult} it is proved that a function $a \in C^{\infty}(\mathbb{R}^d)$ is an $\omega$-multiplier if and only if $af \in \mathcal{S}_{\omega}(\mathbb{R}^d)$ for every $f \in \mathcal{S}_{\omega}(\mathbb{R}^d)$, and also that this is equivalent to the fact that the multiplication operator ${\mathcal M}_a$ is continuous from $\mathcal{S}_{\omega}(\mathbb{R}^d)$ into itself. We also observe that the classical space of multipliers $O_M(\mathbb{R}^d)$ for the Schwartz class  coincides with $O_M^{\omega}(\mathbb{R}^d)$ if $\omega(t)=\log(1+t)$. Now, we give a characterization of the $\omega$-multipliers in terms of the STFT, which generalizes \cite[Proposition 11]{BO}.

\begin{theo}\label{PropCharMultSTFT}
Let $F \in C^{\infty}(\mathbb{R}^d)$ and $0\neq \psi \in \mathcal{S}_{\omega}(\mathbb{R}^d)$. The following assertions are equivalent:
\begin{enumerate}
\item[(1)] $F \in O_M^{\omega}(\mathbb{R}^d)$.
\item[(2)] For all $\lambda>0$ there exist $C_{\lambda}, \mu_{\lambda}>0$ such that
$$ |D^{\alpha}_x D^{\beta}_{\xi} V_{\psi}F(x,\xi)| \leq C_{\lambda} e^{\lambda\varphi^{\ast}_{\omega}\big(\frac{|\alpha|}{\lambda}\big)} e^{\lambda\varphi^{\ast}_{\omega}\big(\frac{|\beta|}{\lambda}\big)} e^{-\lambda\omega(\xi)} e^{\mu_{\lambda}\omega(x)}, \qquad (\alpha,\beta) \in \mathbb{N}_0^{2d}, \ (x,\xi) \in \mathbb{R}^{2d}. $$
\item[(3)] For all $\lambda>0$ there exist $C_{\lambda}, \mu_{\lambda}>0$ such that
$$ |V_{\psi}F(x,\xi)| \leq C_{\lambda} e^{-\lambda\omega(\xi)} e^{\mu_{\lambda}\omega(x)}, \qquad (x,\xi) \in \mathbb{R}^{2d}. $$
\end{enumerate}
\end{theo}
\begin{proof}
$(1) \Rightarrow (2)$ Let us start assuming that $F \in C^{\infty}(\mathbb{R}^d)$ satisfies the estimate~\eqref{EqMultDer}. We distinguish two cases. First, we consider $|\xi|\ge 1$. We integrate by parts and, using Leibniz rule we obtain, for all $N \in \mathbb{N}_0$ and arbitrary $\alpha,\beta \in \mathbb{N}_0^d$, the following:
\begin{align*}
D^{\alpha}_x D^{\beta}_{\xi} V_{\psi}F(x,\xi) &= \int_{\mathbb{R}^d} D^{\beta}_{\xi} (e^{-iy \cdot \xi}) D^{\alpha}_x \overline{\psi(y-x)} F(y) dy \\
&= \int_{\mathbb{R}^d} (-1)^{|\beta|} e^{-iy \cdot \xi} y^{\beta} D^{\alpha}_x \overline{\psi(y-x)} F(y) dy \\
&= (-1)^{|\beta|} (-1)^N |\xi|^{-2N} \int_{\mathbb{R}^d} \Delta_y^N (e^{-iy \cdot \xi}) y^{\beta} D^{\alpha}_x \overline{\psi(y-x)} F(y) dy \\
&= (-1)^{|\beta|} (-1)^N |\xi|^{-2N} \int_{\mathbb{R}^d} e^{-iy \cdot \xi} \Delta_y^N \big( y^{\beta} D^{\alpha}_x \overline{\psi(y-x)} F(y) \big) dy \\
&= (-1)^{|\beta|} |\xi|^{-2N} \int_{\mathbb{R}^d} e^{-iy \cdot \xi} \sum_{|\nu|=N} \frac{N!}{\nu!} D^{2\nu}_y \big( y^{\beta} D^{\alpha}_x \overline{\psi(y-x)} F(y) \big) dy \\
&= (-1)^{|\beta|} |\xi|^{-2N} \int_{\mathbb{R}^d} e^{-iy \cdot \xi} \sum_{|\nu|=N} \frac{N!}{\nu!} \sum_{h \leq \min\{2\nu, \beta\}} \binom{2\nu}{h} \frac{\beta!}{(\beta-h)!} \times \\
& \qquad \times (-i)^hy^{\beta-h} D^{2\nu-h}_y \big( D^{\alpha}_x \overline{\psi(y-x)} F(y) \big) dy.
\end{align*}
Now, we estimate 
$$ \big| D^{2\nu-h}_y \big( D^{\alpha}_x \overline{\psi(y-x)} F(y) \big) \big| \leq \sum_{\gamma \leq 2\nu - h} \binom{2\nu - h}{\gamma} |D^{\alpha}_x D^{\gamma}_y \overline{\psi(y-x)}| |D^{2\nu-h-\gamma}_y F(y)|. $$
From condition $(\gamma)$ of the weight $\omega$, we fix $a\in\mathbb{R}$ and $b>0$ such that
$$ \omega(t) \geq a+b\log(1+t), \qquad t\geq 0. $$
We can also assume, without loss of generality, that
$\omega(et)\leq L\omega(t)+L$
with the same constant as in condition $(\alpha)$ of the weight (this will be useful in the
following to apply \cite[Lemma~A.1]{RealPW}).

By the assumption on $F$, for all $\lambda>0$ there exist $C_{\lambda}>0$ and $\mu_{\lambda}>0$ such that
$$ |D^{2\nu-h-\gamma}_y F(y)| \leq C_{\lambda} e^{(\lambda+2/b) L^{k+1} \varphi^{\ast}_{\omega}\big(\frac{|2\nu-h-\gamma|}{(\lambda+2/b) L^{k+1}}\big)} e^{\mu_{\lambda}\omega(y)}, $$
where $k \in \mathbb{N}_0$ is such that $2 \sqrt{d} \leq e^k$. We fix $\sigma:=(d+1)/b>0$. By~\cite[Th. 2.4(g)]{RealPW}, with the previous $\lambda>0$ there exists $C'_{\lambda}>0$ satisfying
$$ |D^{\alpha}_x D^{\gamma}_y \overline{\psi(y-x)}| \leq C'_{\lambda} e^{2(\lambda+2/b) L^{k+1} \varphi^{\ast}_{\omega}\big(\frac{|\alpha+\gamma|}{2(\lambda+2/b) L^{k+1}}\big)} e^{-(\mu_{\lambda}+\lambda L+\sigma)L\omega(y-x)}. $$
Since $\omega(t)=o(t)$ when $t\to+\infty$, by~\cite[Lemma A.1(v), (viii) and (ii)]{RealPW}, we have that, for the previous $\lambda>0$, there exists $C''_{\lambda}>0$ with
\begin{align*}
\frac{\beta!}{(\beta-h)!} |y|^{|\beta-h|} &\leq h! \binom{\beta}{h} e^{\lambda L \varphi^{\ast}_{\omega}\big(\frac{|\beta-h|}{\lambda L}\big)} e^{\lambda L \omega( y )-\lambda L\varphi^*(0)} \\
&\leq C''_{\lambda} e^{\lambda L \varphi^{\ast}_{\omega}\big(\frac{|h|}{\lambda L}\big)} 2^{|\beta|} e^{\lambda L \varphi^{\ast}_{\omega}\big(\frac{|\beta-h|}{\lambda L}\big)} e^{\lambda L \omega(y)} \leq C''_{\lambda} e^{ \lambda L} e^{\lambda \varphi^{\ast}_{\omega}\big(\frac{|\beta|}{\lambda}\big)} e^{\lambda L \omega(y)}.
\end{align*}
Then, we estimate $|D^{\alpha}_x D^{\beta}_{\xi} V_{\psi}F(x,\xi)|$ by
\begin{align*}
& C_{\lambda} C'_{\lambda} C''_{\lambda} e^{ \lambda L} |\xi|^{-2N} \sum_{|\nu|=N} \frac{N!}{\nu!} \sum_{h \leq \min\{2\nu, \beta\}} \binom{2\nu}{h} e^{\lambda \varphi^{\ast}_{\omega}\big(\frac{|\beta|}{\lambda}\big)} \times \\
& \qquad \times \sum_{\gamma \leq 2\nu - h} \binom{2\nu - h}{\gamma} e^{2(\lambda+2/b) L^{k+1} \varphi^{\ast}_{\omega}\big(\frac{|\alpha+\gamma|}{2(\lambda+2/b) L^{k+1}}\big) + (\lambda+2/b) L^{k+1} \varphi^{\ast}_{\omega}\big(\frac{|2\nu-h-\gamma|}{(\lambda+2/b) L^{k+1}}\big)} \times \\
& \qquad \times \int_{\mathbb{R}^d} e^{-(\mu_{\lambda} + \lambda L + \sigma)L\omega(y-x)+(\mu_{\lambda}+\lambda L)\omega(y)} dy.
\end{align*}
Again by~\cite[Lemma A.1]{RealPW}
\begin{align*}
& \sum_{\gamma \leq 2\nu - h} \binom{2\nu - h}{\gamma} e^{2(\lambda+2/b) L^{k+1} \varphi^{\ast}_{\omega}\big(\frac{|\alpha+\gamma|}{2(\lambda+2/b) L^{k+1}}\big) + (\lambda+2/b) L^{k+1} \varphi^{\ast}_{\omega}\big(\frac{|2\nu-h-\gamma|}{(\lambda+2/b) L^{k+1}}\big)} \\
& \quad \leq e^{\lambda \varphi^{\ast}_{\omega}\big(\frac{|\alpha|}{\lambda}\big)} e^{(\lambda+2/b) L^{k+1} \varphi^{\ast}_{\omega}\big(\frac{|2\nu-h|}{(\lambda+2/b) L^{k+1}}\big)} 2^{|2\nu-h|} \\
& \quad \leq e^{\lambda \varphi^{\ast}_{\omega}\big(\frac{|\alpha|}{\lambda}\big)} e^{(\lambda+2/b) L^k \varphi^{\ast}_{\omega}\big(\frac{2N}{(\lambda+2/b) L^k}\big)} e^{(\lambda+2/b) L^{k+1}}.
\end{align*}
Since $2\sqrt{d} \leq e^k$, we have
\begin{align*}
\sum_{|\nu|=N} \frac{N!}{\nu!} \sum_{h \leq \min\{2\nu, \beta\}} \binom{2\nu}{h} &= \sum_{|\nu|=N} \frac{N!}{\nu!} \sum_{h_1=0}^{\min\{2\nu_1, \beta_1\}} \binom{2\nu_1}{h_1} \cdots \sum_{h_d=0}^{\min\{2\nu_d, \beta_d\}} \binom{2\nu_d}{h_d} \\
&\leq d^N 2^{2N} \leq e^{2kN},
\end{align*}
and hence, by~\cite[Lemma A.1(iii)]{RealPW},
$$ e^{2kN} e^{(\lambda+2/b) L^k \varphi^{\ast}_{\omega}\big(\frac{2N}{(\lambda+2/b) L^k}\big)} \leq e^{(\lambda+2/b) \varphi^{\ast}_{\omega}\big(\frac{2N}{\lambda+2/b}\big)} e^{k(\lambda+2/b) L^k}. $$
Finally, from condition $(\alpha)$ of the weight $\omega$,
$$ (\mu_{\lambda} + \lambda L+ \sigma)\omega(y) \leq (\mu_{\lambda} + \lambda L + \sigma)L \omega(y-x) + (\mu_{\lambda} + \lambda L + \sigma)L \omega(x) + (\mu_{\lambda} + \lambda L + \sigma)L, $$
and therefore
\begin{equation}\label{EqPropCharMultipliers21}
\begin{split}
& -(\mu_{\lambda} + \lambda L + \sigma)L \omega(y-x) + (\mu_{\lambda} + \lambda L)\omega(y) \\
& \qquad \qquad \leq -\sigma\omega(y) + (\mu_{\lambda} + \lambda L + \sigma)L \omega(x) + (\mu_{\lambda} + \lambda L + \sigma)L.
\end{split}
\end{equation}
Thus, for $C'''_{\lambda} = C_{\lambda} C'_{\lambda} C''_{\lambda} e^{\lambda L +  
k(\lambda+2/b)L^k} e^{(\mu_{\lambda} + \lambda L + \sigma)L} > 0$, we deduce
\begin{align*}
|D^{\alpha}_x D^{\beta}_{\xi} V_{\psi}F(x,\xi)| &\leq C'''_{\lambda} \Big( \int_{\mathbb{R}^d} e^{-\sigma\omega(y)} dy \Big) e^{(\mu_{\lambda} + \lambda L+ \sigma)L \omega(x)} e^{\lambda \varphi^{\ast}_{\omega}\big(\frac{|\alpha|}{\lambda}\big)} e^{\lambda \varphi^{\ast}_{\omega}\big(\frac{|\beta|}{\lambda}\big)} \\
& \qquad \times |\xi|^{-2N} e^{(\lambda+2/b) \varphi^{\ast}_{\omega}\big(\frac{2N}{\lambda+2/b}\big)},
\end{align*}
which is valid for all $N \in \mathbb{N}_0$. Hence,  by~\cite[Lemma A.1(vi)]{RealPW}, since $|\xi|\ge 1$, we obtain
$$ \inf_{N \in \mathbb{N}_0} |\xi|^{-2N} e^{(\lambda+2/b)\varphi^{\ast}_{\omega}\big(\frac{2N}{\lambda+2/b}\big)} \leq e^{-\lambda\omega(\xi)} e^{-a\frac{2}{b}}. $$
Therefore, for $\mu'_{\lambda} = (\mu_{\lambda}+\lambda L+\sigma)L>0$, we conclude
$$ |D^{\alpha}_x D^{\beta}_{\xi} V_{\psi}F(x,\xi)| \leq C'''_{\lambda} e^{-\frac{2a}{b}} \Big( \int_{\mathbb{R}^d} e^{-\sigma\omega(y)} dy \Big) e^{\lambda \varphi^{\ast}_{\omega}\big(\frac{|\alpha|}{\lambda}\big)} e^{\lambda \varphi^{\ast}_{\omega}\big(\frac{|\beta|}{\lambda}\big)} e^{-\lambda\omega(\xi)} e^{\mu'_{\lambda} \omega(x)}. $$
This shows the result for the case $|\xi|\ge 1$  since the integral above converges by condition $(\gamma)$ of the weight function $\omega$ and by the choice of $\sigma$ (see, for instance,~\cite[(2.3)]{RealPW}).

For the case when $|\xi|\le 1$ we can argue similarly, but the proof is much easier. In fact, the integration by parts to obtain the estimate by $e^{-\lambda \omega(\xi)}$ is not needed since, in this case, such an estimate is 
straightforward, because
$1=
e^{\lambda \omega(\xi)} e^{-\lambda \omega(\xi)}\le D_\lambda e^{-\lambda \omega(\xi)}
$, for $D_\lambda:=\max_{|\xi|\le 1}e^{\lambda \omega(\xi)}.$ 

$(2) \Rightarrow (3)$ Trivial.

$(3) \Rightarrow (1)$ If $F$ satisfies $(3)$, we first point out that the inversion formula for the short-time Fourier transform~\cite[Lemma 1.1]{GZ} holds. Indeed, if $\gamma \in \mathcal{S}_{\omega}(\mathbb{R}^d)\setminus\{0\}$, then it is easy to see (using~\cite[Th 2.4]{GZ}) that
$$ \int_{\mathbb{R}^{2d}} V_{\psi}F(z) \Pi(z)\gamma dz $$
defines a convergent integral. Hence, if $\langle \gamma, \psi \rangle \neq 0$ and proceeding as in~\cite[Lemma 1.1]{GZ} we have 
\begin{equation}\label{EqInvSTFTNew}
F = (2\pi)^{-d} \langle \gamma, \psi \rangle^{-1} \int_{\mathbb{R}^{2d}} V_{\psi}F(z) \Pi(z)\gamma dz.
\end{equation}
Now, for $\gamma=\psi$ and denoting $C=(2\pi)^{-d} \norm{\psi}^{-2}_{L^2(\mathbb{R}^d)}>0$,
$$ D^{\alpha} F(y) = C \int_{\mathbb{R}^{2d}} V_{\psi}F(x,\xi) \sum_{\alpha' \leq \alpha} \binom{\alpha}{\alpha'} \xi^{\alpha'} e^{i y \cdot \xi} D^{\alpha-\alpha'}_y \overline{\psi(y-x)} dx d\xi, \quad \alpha \in \mathbb{N}_0^d, \ y \in \mathbb{R}^d. $$
Take $\sigma>0$ as in the proof of $(1) \Rightarrow (2)$. By assumption, for all $\lambda>0$ there exist $C_{\lambda}=C_{\lambda L+\sigma}>0$ and $\mu_{\lambda}=\mu_{\lambda L+\sigma}>0$ such that
$$ |V_{\psi}F(x,\xi)| \leq C_{\lambda} e^{-(\lambda L + \sigma) \omega(\xi)} e^{\mu_{\lambda}\omega(x)}, \qquad x,\xi \in \mathbb{R}^d. $$
For that $\lambda>0$, we use~\cite[Lemma A.1(v)]{RealPW} again as follows:
$$ |\xi^{\alpha'}| \leq |\xi|^{|\alpha'|}  \leq e^{\lambda L \varphi^{\ast}_{\omega}\big(\frac{|\alpha'|}{\lambda L}\big)} e^{\lambda L \omega( \xi )} e^{-\lambda L \varphi^*(0)}. $$
Moreover, since $\psi \in \mathcal{S}_{\omega}(\mathbb{R}^d)$, there exists $C'_{\lambda}=C'_{\lambda L, \mu_{\lambda}+\sigma}>0$ such that
$$ |D^{\alpha-\alpha'}_y \overline{\psi(y-x)}| \leq C'_{\lambda} e^{\lambda L \varphi^{\ast}_{\omega}\big(\frac{|\alpha-\alpha'|}{\lambda L}\big)} e^{-(\mu_{\lambda}+\sigma)L \omega(y-x)}, \qquad\mbox{for}\ \alpha' \leq \alpha. $$
So, the modulus of the derivatives of $F$ is estimated by
\begin{align*}
|D^{\alpha} F(y)| &\leq C C_{\lambda} C'_{\lambda} e^{-\lambda L \varphi^*(0)}\Big( \sum_{\alpha' \leq \alpha} \binom{\alpha}{\alpha'} e^{\lambda L \varphi^{\ast}_{\omega}\big(\frac{|\alpha'|}{\lambda L}\big)} e^{\lambda L \varphi^{\ast}_{\omega}\big(\frac{|\alpha - \alpha'|}{\lambda L}\big)} \Big) \times \\
& \quad \times \int_{\mathbb{R}^{2d}} e^{-\sigma\omega(\xi)} e^{\mu_{\lambda} \omega(x)} e^{-(\mu_{\lambda}+\sigma)L\omega(y-x)} dxd\xi.
\end{align*}
Now,~\cite[Lemma A.1(ii)]{RealPW} yields
$$ \sum_{\alpha' \leq \alpha} \binom{\alpha}{\alpha'} e^{\lambda L \varphi^{\ast}_{\omega}\big(\frac{|\alpha'|}{\lambda L}\big)} e^{\lambda L \varphi^{\ast}_{\omega}\big(\frac{|\alpha - \alpha'|}{\lambda L}\big)} \leq e^{\lambda \varphi^{\ast}_{\omega}\big(\frac{|\alpha|}{\lambda}\big)} e^{\lambda L}. $$
As in~\eqref{EqPropCharMultipliers21}, we obtain
$$ -(\mu_{\lambda}+\sigma)L \omega(y-x) + \mu_{\lambda} \omega(x) \leq -\sigma\omega(x) + (\mu_{\lambda}+\sigma)L \omega(y) + (\mu_{\lambda}+\sigma)L. $$
Thus, the integral converges and we conclude that for all $\lambda>0$ there exist $\mu'_{\lambda}=(\mu_{\lambda}+\sigma)L>0$ and $C''_{\lambda}>0$ 
such that
$$ |D^{\alpha} F(y)| \leq C''_{\lambda} e^{\lambda \varphi^{\ast}_{\omega}\big(\frac{|\alpha|}{\lambda}\big)} e^{\mu'_{\lambda}\omega(y)}, \qquad \alpha \in \mathbb{N}_0^d, \  y \in \mathbb{R}^d. $$
\end{proof}
Notice that we have shown that the inversion formula for the short-time Fourier transform~\eqref{EqInvSTFTNew} holds for $F \in O_M^{\omega}(\mathbb{R}^d)$.

Now, we turn our attention into the space of \emph{$\omega$-convolutors in $\mathcal{S}_{\omega}(\mathbb{R}^d)$}, denoted by $(O^{\omega}_C)'(\mathbb{R}^d)$, consisting of those $a \in \mathcal{S}'_{\omega}(\mathbb{R}^d)$ such that $a \ast f \in \mathcal{S}_{\omega}(\mathbb{R}^d)$ for every $f \in \mathcal{S}_{\omega}(\mathbb{R}^d)$. This class has been studied in~\cite{AM}. Again, we write $(O_C)'(\mathbb{R}^d) = (O_C^{\omega})'(\mathbb{R}^d)$ when $\omega(t)=\log(1+t)$.

We give a characterization of the $\omega$-convolutors in terms of the STFT, extending the result given for $\mathcal{S}(\mathbb{R}^d)$ in~\cite[Proposition 10]{BO} to the ultradifferentiable setting. To this aim, we use the characterization given for the $\omega$-multipliers in Theorem~\ref{PropCharMultSTFT}. 
\begin{theo}\label{TheoCharConvSTFT}
For $a \in \mathcal{S}'_{\omega}(\mathbb{R}^d)$ and $0 \neq \psi \in \mathcal{S}_{\omega}(\mathbb{R}^d)$, it is equivalent:
\begin{enumerate}
\item[(1)] $a \in (O_C^{\omega})'(\mathbb{R}^d)$.
\item[(2)] For every $\mu>0$ there exist $C_{\mu}, \lambda_{\mu}>0$ such that
$$ |D^{\alpha}_x D^{\beta}_{\xi} V_{\psi}a(x,\xi)| \leq C_{\mu} e^{\mu \varphi^{\ast}_{\omega}\big(\frac{|\alpha|}{\mu}\big)} e^{\mu \varphi^{\ast}_{\omega}\big(\frac{|\beta|}{\mu}\big)} e^{-\mu\omega(x)} e^{\lambda_{\mu}\omega(\xi)}, \qquad (\alpha,\beta) \in \mathbb{N}_0^{2d}, \ (x,\xi) \in \mathbb{R}^{2d}. $$
\item[(3)] For every $\mu>0$ there exist $C_{\mu}, \lambda_{\mu}>0$ such that
$$ |V_{\psi}a(x,\xi)| \leq C_{\mu} e^{-\mu\omega(x)} e^{\lambda_{\mu}\omega(\xi)}, \qquad (x,\xi) \in \mathbb{R}^{2d}. $$
\end{enumerate}
\end{theo}
\begin{proof}
$(1) \Leftrightarrow (2)$ By definition of convolutor and since $\mathcal{S}_{\omega}(\mathbb{R}^d)$ is invariant by Fourier transform, we have from~\cite[Kap. I, \textsection{6}]{Fieker} that $a$ is an $\omega$-convolutor if and only if $\widehat{a \ast f} = \widehat{a} \cdot \widehat{f} \in \mathcal{S}_{\omega}(\mathbb{R}^d)$ for every $f \in \mathcal{S}_{\omega}(\mathbb{R}^d)$, i.e., $\widehat{a} \in O_M^{\omega}(\mathbb{R}^d)$, since the Fourier transform is an isomorphism in $\mathcal{S}_{\omega}(\mathbb{R}^d)$. This is equivalent, by Theorem~\ref{PropCharMultSTFT}, to the condition that for all $\lambda>0$ there exist $C_{\lambda},\mu_{\lambda}>0$ such that
$$ |D^{\alpha}_x D^{\beta}_{\xi} V_{\widehat{\psi}}\widehat{a}(x,\xi)| \leq C_{\lambda} e^{\lambda \varphi^{\ast}_{\omega}\big(\frac{|\alpha|}{\lambda}\big)} e^{\lambda\varphi^{\ast}_{\omega}\big(\frac{|\beta|}{\lambda}\big)} e^{-\lambda\omega(\xi)} e^{\mu_{\lambda}\omega(x)}, \qquad (\alpha,\beta) \in \mathbb{N}_0^{2d}, \ (x,\xi) \in \mathbb{R}^{2d}. $$
Since $|D^{\alpha}_x D^{\beta}_{\xi} V_{\widehat{\psi}}\widehat{a}(x,\xi)| = |D^{\alpha}_x D^{\beta}_{\xi} V_{\psi}a(-\xi,x)|$ and the weight $\omega$ is radial, we obtain the inequality on the statement $(2)$. For the proof $(1) \Leftrightarrow (3)$ we can proceed similarly. 
\end{proof}


\section{Compactness of the multiplication operator}\label{SectCompMult}
A particular case of Weyl operator is the multiplication operator (see Section~\ref{Weyl oper} below). In this section,  we show that such an operator  is compact only when it is the multiplication by zero. First, we recall that the \emph{spectrum} of a linear operator $T:E\to E$ defined on a Hausdorff locally convex  space $E$ is the set
$$ \sigma(T; E) := \{ \lambda \in \mathbb{C}: \ \lambda \cdot I - T \  \text{is not invertible} \}. $$


\begin{theo}\label{TheoCompMult}
The multiplication operator ${\mathcal M}_a:\mathcal{S}_{\omega}(\mathbb{R}^d)\to \mathcal{S}_{\omega}(\mathbb{R}^d)$ is compact if and only if $a=0$.
\end{theo}
\begin{proof}
We consider $a\in O_M^{\omega}(\mathbb{R}^d)$ (if not, ${\mathcal M}_a$ is not continuous, nor compact by \cite[Theorem~4.4]{AM-Mult}). In particular, $a$ is continuous. First, we check that the image of $a$ is contained in the spectrum of the multiplication operator ${\mathcal M}_a:\mathcal{S}_{\omega}(\mathbb{R}^d)\to \mathcal{S}_{\omega}(\mathbb{R}^d)$, i.e.
\begin{equation}\label{EqDani}
a(\mathbb{R}^d) \subseteq \sigma({\mathcal M}_a; \mathcal{S}_{\omega}(\mathbb{R}^d)).
\end{equation}
Indeed, fix $x \in \mathbb{R}^d$. Then, for each function $g$ in the image of $\{ a(x) I - {\mathcal M}_a \}$, we have  $g(x) = a(x) g(x) - a(x)g(x) = 0$. Since there exist $f \in \mathcal{S}_{\omega}(\mathbb{R}^d)$ such that $f(x) \neq 0$, we obtain that $a(x)I - {\mathcal M}_a$ cannot be surjective. Therefore $a(x) \in \sigma({\mathcal M}_a; \mathcal{S}_{\omega}(\mathbb{R}^d))$.

If ${\mathcal M}_a$ is compact, then by~\cite[Ch. 5, Part 2, Th. 4]{Grot}, the set $\sigma({\mathcal M}_a; \mathcal{S}_{\omega}(\mathbb{R}^d))$ is either finite or the closure of a sequence that tends to zero. We claim that the function $a$ is constant. If not, fix $x_0\in\mathbb{R}^d$ and $\varepsilon>0$ small enough such that $|a|$ is always positive in $\overline{B(x_0,\varepsilon)}$. If $m<M$ are the minimum and maximum value attained by $|a|$ in $\overline{B(x_0,\varepsilon)}$ at $x_1, x_2 \in \mathbb{R}^d$, then  the function
$$ g(\lambda) = |a|((1-\lambda) x_1 + \lambda x_2),\ \ 0\le \lambda \le 1,   $$
satisfies $g([0,1])=[m,M]$ (by the intermediate value theorem), which is contained in the image of the function $|a|$. This is a contradiction with the  inclusion~\eqref{EqDani}. Hence, the multiplication operator becomes the multiplication by a complex number, and is compact only in the case that this complex number  is zero. 
\end{proof}

\begin{cor}
The multiplication operator ${\mathcal M}_a:\mathcal{S}(\mathbb{R}^d) \to \mathcal{S}(\mathbb{R}^d)$ is compact if and only if $a=0$.
\end{cor}

Given $a \in (O_C^{\omega})'(\mathbb{R}^d)$, the \emph{convolution operator} ${\mathcal C}_a:\mathcal{S}_{\omega}(\mathbb{R}^d)\to \mathcal{S}_{\omega}(\mathbb{R}^d)$ defined by ${\mathcal C}_a(f):=a\ast f$, for each $f\in \mathcal{S}_{\omega}(\mathbb{R}^d)$, is well defined and continuous. Moreover,  if $\log(t)=o(\omega(t))$ as $t\to\infty$, the converse also holds~\cite[Theorem 5.3]{AM}. Regarding compactness, we have 
\begin{theo}\label{TheoCompConv}
Given $a \in (O_C^{\omega})'(\mathbb{R}^d)$, the convolution operator ${\mathcal C}_a:\mathcal{S}_{\omega}(\mathbb{R}^d) \to \mathcal{S}_{\omega}(\mathbb{R}^d)$ is compact if and only if $a=0$.
\end{theo}
\begin{proof}
Since $\mathcal{F}$ is a topological isomorphism in $\mathcal{S}_{\omega}(\mathbb{R}^d)$, we have ${\mathcal C}_a$ is compact if and only if $\mathcal{F} \circ {\mathcal C}_a$ is also compact. Since $a\in \mathcal{S}_{\omega}'(\mathbb{R}^d)$ (see, for instance, the comments after Theorem 3.2 in \cite{AM}), by~\cite[Kap. I, \textsection{6}, Satz~6.10]{Fieker}, we have $\mathcal{F}({\mathcal C}_a(f)) = {\mathcal M}_{\mathcal{F}(a)}(\widehat{f})$ for every $f\in\mathcal{S}_{\omega}(\mathbb{R}^d)$. Hence, the conclusion follows by Theorem~\ref{TheoCompMult}.
\end{proof}

\section{Continuity and compactness of the Weyl operator}\label{Weyl oper}

We begin with the following observation.

\begin{lema}
	\label{well-def}
	If $X$ is a Fr\'echet space, any continuous linear operator $T:X\to X'$ such that $T(X)\subseteq X$ is continuous as an operator $T:X\to X$.
\end{lema}

\begin{proof}
	It is enough to apply the closed graph theorem for Fr\'echet spaces.
\end{proof}
As a consequence, we obtain 
\begin{prop}\label{PropClosedGraph}
Let $a \in \mathcal{S}'_{\omega}(\mathbb{R}^{2d})$. The Weyl operator $a^w(x,D):\mathcal{S}_{\omega}(\mathbb{R}^d)\to\mathcal{S}_{\omega}(\mathbb{R}^d)$ is well defined if and only if it is continuous.
\end{prop}
\begin{proof}
From \eqref{EqWig} it is easy to deduce that the Weyl operator as defined in \eqref{Weyl-oper} is well defined and continuous from $\mathcal{S}_{\omega}(\mathbb{R}^d)$ into $\mathcal{S}'_{\omega}(\mathbb{R}^d)$. 
\end{proof}
Now, we give two elementary, but clarifying examples.
\begin{exam}
(a) If $a=a(x)$ with $a \in \mathcal{S}'_{\omega}(\mathbb{R}^d)$, we have $ a^w(x) = {\mathcal M}_a. $
Indeed, for $f,g \in \mathcal{S}_{\omega}(\mathbb{R}^d)$,
$$ \langle a^w(x)f, g \rangle = (2\pi)^{-d} \langle a \otimes 1, \Wig(g,f) \rangle = (2\pi)^{-d} \langle a(x), \int \Wig(g,f)(x,\xi) d\xi \rangle. $$
From the fact that $\int \widehat{F}(\xi) d\xi = \int e^{i\xi \cdot 0} \widehat{F}(\xi) d\xi = (2\pi)^d F(0)$ for $F \in \mathcal{S}_{\omega}(\mathbb{R}^d)$, we obtain
$$ \int \Wig(g,f)(x,\xi) d\xi = \int \mathcal{F}_{t \mapsto \xi} \left( g\Big(x+\frac{t}{2}\Big) \overline{f\Big(x-\frac{t}{2}\Big)}\right) d\xi = (2\pi)^d g(x) \overline{f(x)}. $$
Therefore,
$$ \langle a^w(x)f, g \rangle = \langle a(x), g(x)\overline{f(x)} \rangle = \langle a(x)f(x), g(x) \rangle, $$
which shows that $a^w(x) = {\mathcal M}_a$. 

So, in this case, the operator $ a^w(x): \mathcal{S}_{\omega}(\mathbb{R}^d) \to \mathcal{S}_{\omega}(\mathbb{R}^d)$ is continuous if and only if $a \in O_M^{\omega}(\mathbb{R}^d)$
and, by Theorem~\ref{TheoCompMult}, 
$ a^w(x): \mathcal{S}_{\omega}(\mathbb{R}^d) \to \mathcal{S}_{\omega}(\mathbb{R}^d)$ is compact if and only if $ a=0. $

(b) Similarly, if we take $a=a(\xi) \in \mathcal{S}'_{\omega}(\mathbb{R}^d)$ it is easy to see that
$ a^w(D)f = \mathcal{F}^{-1}(a) \ast f. $
 Indeed
\begin{align*}
\langle a^w(D)f,g \rangle &= (2\pi)^{-d} \langle 1 \otimes a, \Wig(g,f) \rangle = \langle 1 \otimes \mathcal{F}^{-1}(a), \mathcal{F}^{-1}_2 (\Wig(g,f)) \rangle \\
&= \langle 1_x \otimes \mathcal{F}^{-1}(a)(t), g(x+t/2) \overline{f(x-t/2)} \rangle \\
&= \langle \mathcal{F}^{-1}(a)(t), \int g(x+t/2) \overline{f(x-t/2)} dx \rangle = \langle \mathcal{F}^{-1}(a)(t), \int g(s) \overline{f(s-t)} ds \rangle \\
&= \langle \mathcal{F}^{-1}(a)(t), g \ast \overline{\mathcal{I}f(t)} \rangle = \langle \mathcal{F}^{-1}(a) \ast f, g \rangle.
\end{align*}
This shows, by Theorem~\ref{TheoCompConv}, that
$ a^w(D): \mathcal{S}_{\omega}(\mathbb{R}^d) \to \mathcal{S}_{\omega}(\mathbb{R}^d)$ is compact if and only if $a=0$
and, for weights satisfying $\log(t)=o(\omega(t))$, $t\to\infty$, we have, using~\cite[Theorem~6.1]{AM}, that $a^w(D): \mathcal{S}_{\omega}(\mathbb{R}^d) \to \mathcal{S}_{\omega}(\mathbb{R}^d)$ is continuous if and only if $\mathcal{F}^{-1}(a) \in (O_C^{\omega})'(\mathbb{R}^d)$ if and only if $ a \in O_M^{\omega}(\mathbb{R}^d). $
\end{exam}
In general, the continuity of the Weyl operator in $\mathcal{S}_{\omega}(\mathbb{R}^d)$ is not characterized by the $\omega$-multipliers:
\begin{exam}\label{ExamMultCont}
$(a)$ Consider $a(x,\xi) = \delta_x \otimes a_2(\xi)$ with $a_2 \in (O_C^{\omega})'(\mathbb{R}^d)$. We have
\begin{align*}
\langle a^w(x,D)f, g \rangle &= (2\pi)^{-d} \langle \delta_x \otimes a_2, \Wig(g,f) \rangle \\
&= \langle \delta_x \otimes \mathcal{F}^{-1}(a_2)(t), g(x+t/2) \overline{f(x-t/2)} \rangle \\
&= \langle \mathcal{F}^{-1}(a_2)(t), g(t/2) \overline{f(-t/2)} \rangle \\
&= 2^d \langle \mathcal{F}^{-1}(a_2)(2s), g(s) \overline{\mathcal{I}f(s)} \rangle \\
&= \langle 2^d \mathcal{F}^{-1}(a_2)(2s) \mathcal{I}f(s), g(s) \rangle.
\end{align*}
Therefore,
$$ a^w(x,D)f(t) = 2^d \mathcal{F}^{-1}(a_2)(2t) \mathcal{I}f(t). $$
So, since $a_2 \in (O_C^{\omega})'(\mathbb{R}^d)$, we have $\mathcal{F}^{-1}(a_2) \in O_M^{\omega}(\mathbb{R}^d)$ and hence $a^w(x,D):\mathcal{S}_{\omega}(\mathbb{R}^d) \to \mathcal{S}_{\omega}(\mathbb{R}^d)$ continuously. However, 
$a=\delta_x \otimes a_2(\xi) \notin O_M^{\omega}(\mathbb{R}^{2d})$ for every $0 \neq a_2 \in (O_C^{\omega})'(\mathbb{R}^d)$. 

Very similarly, in the case $a(x,\xi) = a_1(x) \otimes \delta_{\xi}$ with $a_1 \in (O_C^{\omega})'(\mathbb{R}^d)$, it is easy to see, denoting $\Lambda_2a_1(x) = a_1(x/2)$, that $ a^w(x,D)f = \Lambda_2a_1 \ast \mathcal{I}f. $
Indeed,
\begin{align*}
\langle a^w(x,D)f, g \rangle &= (2\pi)^{-d} \langle a_1(x) \otimes \delta_{\xi}, \Wig(g,f) \rangle \\
&= \langle a_1(x) \otimes \mathcal{F}^{-1}(\delta_{\xi}), \mathcal{F}^{-1}_2(\Wig(g,f)) \rangle \\
&= \langle a_1(x) \otimes 1_{t}, g(x+t/2) \overline{f(x-t/2)} \rangle \\
&= \langle a_1(x), \int g(x+t/2) \overline{f(x-t/2)} dt \rangle \\
&= 2^d \langle a_1(x), \int g(s) \overline{f(2x-s)} ds \rangle \\
&= \langle \Lambda_2 a_1(x), \int g(s) \overline{f(x-s)} ds \rangle = \langle \Lambda_2 a_1 \ast \mathcal{I}f, g\rangle,
\end{align*}
from where it follows that $a^w(x,D)f \in \mathcal{S}_{\omega}(\mathbb{R}^d)$ for every $f \in \mathcal{S}_{\omega}(\mathbb{R}^d)$ and $a^w(x,D):\mathcal{S}_{\omega}(\mathbb{R}^d) \to \mathcal{S}_{\omega}(\mathbb{R}^d)$ continuously. However, $a$ is not an $\omega$-multiplier for $0 \neq a_1 \in (O_C^{\omega})'(\mathbb{R}^d)$.

$(b)$ Consider $a(x,\xi) = e^{-2ix\cdot \xi} f(\xi)$ for some $f \in \mathcal{S}_{\omega}(\mathbb{R}^d)\setminus\{0\}$. 
We check that $a \in O_M^{\omega}(\mathbb{R}^{2d})$, i.e. $aF \in \mathcal{S}_{\omega}(\mathbb{R}^{2d})$ for every $F \in \mathcal{S}_{\omega}(\mathbb{R}^{2d})$. In fact, $(aF)(x,\xi) = e^{-2ix\cdot \xi}f(\xi) F(x,\xi) \in \mathcal{S}_{\omega}(\mathbb{R}^{2d})$ if and only if
\begin{align*}
\mathcal{F}_{x \mapsto y} \Big( e^{-2ix\cdot \xi}f(\xi) F(x,\xi) \Big) &= f(\xi) \int e^{-ix\cdot y} e^{-2ix\cdot \xi} F(x,\xi) dx \\
&= f(\xi) \mathcal{F}_1(F)(y+2\xi, \xi)
\end{align*}
belongs to $\mathcal{S}_{\omega}(\mathbb{R}^{2d})$. We have that $G(x,\xi) = \mathcal{F}_1(F)(x+2\xi,\xi) \in \mathcal{S}_{\omega}(\mathbb{R}^{2d})$ since the linear change of coordinates is invertible and $\mathcal{F}_1(F) \in \mathcal{S}_{\omega}(\mathbb{R}^{2d})$. Hence, $G(x,\xi) f(\xi) \in \mathcal{S}_{\omega}(\mathbb{R}^{2d})$, and from this we obtain that $a\in O_M^\omega(\mathbb{R}^{2d})$.

On the other hand, we  see that, in this case,  $a^w(x,D)$ does not take values in 
$\mathcal{S}_{\omega}(\mathbb{R}^d)$. In fact, for $g\in \mathcal{S}_{\omega}(\mathbb{R}^d)$,
\begin{align*}
a^w(x,D)g &= (2\pi)^{-d} \int_{\mathbb{R}^{2d}} e^{i(x-s)\cdot \xi} a\Big(\frac{x+s}{2}, \xi\Big) g(s) ds d\xi \\
&= (2\pi)^{-d} \int_{\mathbb{R}^{2d}} e^{i(x-s)\cdot \xi} e^{-2i \frac{x+s}{2} \cdot \xi} f(\xi) g(s) ds d\xi \\
&= (2\pi)^{-d} \int_{\mathbb{R}^{2d}} e^{-i(2s)\cdot\xi} f(\xi) g(s) ds d\xi \\
&= (2\pi)^{-d} \int_{\mathbb{R}^{2d}} \widehat{f}(2s) g(s) ds = C_g,
\end{align*}
a constant. Since there exist $g \in \mathcal{S}_{\omega}(\mathbb{R}^d)$ satisfying $C_g\neq 0$ (for example, $g(s) = \widehat{f}(2s)$), we have, for such a $g$, that $a^w(x,D)g\notin \mathcal{S}_{\omega}(\mathbb{R}^d)$.
\end{exam}

Since now we deal with modulation spaces with exponential weights that depend on weight functions, we assume that the weight functions $\omega$ are subadditive.  In order to characterize the continuity and compactness of the Weyl operator in $\mathcal{S}_{\omega}(\mathbb{R}^d)$ via Proposition~\ref{LemmaKEY}, we first give a version of~\cite[Prop. 4.7]{ToftCont} (see also~\cite{ToftBarg}) for modulation spaces with exponential weights.


\begin{prop}\label{PropConverse}
Let $\omega$ be a subadditive weight function and assume that $T$ is a linear and continuous map from $\mathcal{S}_{\omega}(\mathbb{R}^{d_1})$ into $\mathcal{S}'_{\omega}(\mathbb{R}^{d_2})$ and fix $\lambda,\mu>0$. The operator $T$ extends (uniquely) to a continuous mapping from $M^1_{\mu}(\mathbb{R}^{d_1})$ into $M^{\infty}_{\lambda}(\mathbb{R}^{d_2})$ if and only if there exists $K \in \mathcal{S}'_{\omega}(\mathbb{R}^{d_2+d_1})$ such that
$$ K \in M^{\infty}_{\lambda \otimes (-\mu)}(\mathbb{R}^{d_2+d_1}), $$
that is, given $\psi \in \mathcal{S}_{\omega}(\mathbb{R}^{d_2+d_1}) \setminus \{0\}$,
\begin{equation}\label{EqEquivalenceKernel}
\esup_{(x,y,\xi,\eta) \in \mathbb{R}^{2d_2+2d_1}} |V_{\psi} K(x,y,\xi,\eta)| e^{\lambda \omega(x,\xi)} e^{-\mu \omega(y,\eta)} < +\infty,
\end{equation}
and satisfying 
\begin{equation}\label{EqPropKernel}
(Tf)(x) = \langle K(x,\cdot), \overline{f} \rangle, \qquad f \in \mathcal{S}_{\omega}(\mathbb{R}^{d_1}).
\end{equation}
\end{prop}

\begin{proof}
First, we assume that there exists the extension of  $T: M^1_{\mu}(\mathbb{R}^{d_1})\to M^{\infty}_{\lambda}(\mathbb{R}^{d_2})$. Since $\mathcal{S}_{\omega}(\R^d)$ is nuclear~(see, e.g. \cite{BJOS-2021,BJOS-2020}), there is (a unique) $K \in \mathcal{S}'_{\omega}(\mathbb{R}^{d_2+d_1})$ satisfying~\eqref{EqPropKernel}. We show that the estimate~\eqref{EqEquivalenceKernel} holds. 
From the continuity of the extension of $T$, it follows that there exists $C>0$ 
such that 
\begin{equation}\label{EqKgf}
|\langle K, g \otimes \overline{f} \rangle| = |\langle Tf, g\rangle| \leq \norm{Tf}_{M^{\infty}_{\lambda}(\mathbb{R}^{d_2})} \norm{g}_{M^1_{-\lambda}(\mathbb{R}^{d_2})} \leq C \norm{f}_{M^1_{\mu}(\mathbb{R}^{d_1})} \norm{g}_{M^{1}_{-\lambda}(\mathbb{R}^{d_2})},
\end{equation}
for all $f \in \mathcal{S}_{\omega}(\mathbb{R}^{d_1})$ and $g \in \mathcal{S}_{\omega}(\mathbb{R}^{d_2})$. Consider $\psi = g \otimes \overline{f} \in \mathcal{S}_{\omega}(\mathbb{R}^{d_2+d_1})$. We replace $f$ and $g$ in~\eqref{EqKgf} with
$$ f_{y,\eta}(t) := \Pi(y,-\eta)f(t) = e^{-it\cdot \eta} f(t-y) \quad \text{and} \quad g_{x,\xi}(t) := \Pi(x,\xi)g(t) = e^{it\cdot \xi} g(t-x) $$
and we have 
\begin{align*}
\langle K, g_{x,\xi} \otimes \overline{f_{y,\eta}} \rangle 
&= \int K(s,t) e^{-is\cdot\xi} \overline{g(s-x)} e^{-it\cdot\eta} f(t-y) ds dt \\
&= \int K(s,t) \overline{\psi(s-x,t-y)} e^{-i(s,t)\cdot(\xi,\eta)} ds dt = V_{\psi}K(x,y,\xi,\eta).
\end{align*}
Therefore, from~\eqref{EqKgf}, we obtain
\begin{equation}\label{EqEstKernel}
|V_{\psi}K(x,y,\xi,\eta)| \leq C \norm{f_{y,\eta}}_{M^1_{\mu}(\mathbb{R}^{d_1})} \norm{g_{x,\xi}}_{M^{1}_{-\lambda}(\mathbb{R}^{d_2})}.
\end{equation}
Now, we estimate the right-hand side of~\eqref{EqEstKernel}. First, we compute,
for $\psi_1\in{\mathcal S}_\omega(\R^{d_1})\setminus\{0\}$,
\begin{align*}
V_{\psi_1} f_{y,\eta}(u,v) 
&= \int_{\mathbb{R}^{d_1}} e^{-it\cdot\eta} f(t-y) \overline{\psi_1(t-u)} e^{-it\cdot v} dt \\
&= \int_{\mathbb{R}^{d_1}} e^{-it\cdot(v+\eta)} f(t-y) \overline{\psi_1(t-u)} dt \\
&= e^{-iy\cdot(v+\eta)} V_{\psi_1}f(u-y, v+\eta), \qquad (u,v) \in \mathbb{R}^{2d_1}.
\end{align*}
Therefore
\begin{align*}
\norm{f_{y,\eta}}_{M^1_{\mu}(\mathbb{R}^{d_1})} &= \int_{\mathbb{R}^{2d_1}} |V_{\psi_1}f(u-y,v+\eta)| e^{\mu \omega(u,v)} du dv \\
&= \int_{\mathbb{R}^{2d_1}} |V_{\psi_1}f(u,v)| e^{\mu \omega(u+y, v-\eta)} du dv.
\end{align*}
Since $\mu>0$, then we estimate $\mu \omega(u+y,v-\eta) \leq \mu\omega(u,v) + \mu \omega(y,\eta)$, so we have
$$ \norm{f_{y,\eta}}_{M^1_{\mu}(\mathbb{R}^{d_1})} \leq e^{\mu \omega(y,\eta)} \norm{f}_{M^1_{\mu}(\mathbb{R}^{d_1})}. $$
Thus, as $f \in \mathcal{S}_{\omega}(\mathbb{R}^{d_1})$, there exists $C_1>0$ such that
\begin{equation}\label{Eq1}
\norm{f_{y,\eta}}_{M^1_{\mu}(\mathbb{R}^{d_1})} \leq C_1 e^{\mu \omega(y,\eta)}.
\end{equation}
We proceed similarly for $g_{x,\xi}$: given $\psi_2 \in \mathcal{S}_{\omega}(\mathbb{R}^{d_2}) \setminus \{0\}$, one can see
\begin{align*}
V_{\psi_2} g_{x,\xi}(u,v) 
= e^{-ix\cdot(v-\xi)} V_{\psi_2}g(u-x,v-\xi), \qquad (u,v) \in \mathbb{R}^{2d_2}.
\end{align*}
So, we have
\begin{align*}
\norm{g_{x,\xi}}_{M^1_{-\lambda}(\mathbb{R}^{d_2})} 
&= \int_{\mathbb{R}^{d_2}} |V_{\psi_2}g(u,v)| e^{-\lambda \omega(u+x,v+\xi)} du dv.
\end{align*}
For $\lambda>0$, we have $-\lambda\omega(u+x,v+\xi) \leq -\lambda\omega(x,\xi) + \lambda\omega(u,v)$. Thus, there exists $C_2>0$ such that
\begin{equation}\label{Eq2}
\norm{g_{x,\xi}}_{M^1_{-\lambda}(\mathbb{R}^{d_2})} \leq e^{-\lambda \omega(x,\xi)} \norm{g}_{M^1_{\lambda}(\mathbb{R}^{d_2})} \leq C_2 e^{-\lambda \omega(x,\xi)}.
\end{equation}
Substituting~\eqref{Eq1} and~\eqref{Eq2} into formula~\eqref{EqEstKernel}, we find $C'>0$ such that
$$ |V_{\psi} K(x,y,\xi,\eta)| \leq C' e^{-\lambda \omega(x,\xi)} e^{\mu \omega(y,\eta)}, $$
and hence $K \in M^{\infty}_{\lambda \otimes (-\mu)}(\mathbb{R}^{d_2+d_1})$.


Since $M^1_{\mu}(\mathbb{R}^{d_1})$ is a Banach space and the inclusion $\mathcal{S}_{\omega}(\mathbb{R}^{d_1}) \hookrightarrow M^1_{\mu}(\mathbb{R}^{d_1})$ is dense and continuous, the equality $\langle K, g \otimes \overline{f} \rangle = \langle Tf, g \rangle$ in $\mathcal{S}_{\omega}(\mathbb{R}^{d_1})$ can be extended to the modulation spaces $M^1_{\mu}(\mathbb{R}^{d_1})$ (see for instance the comment in the proof of~\cite[Theorem~14.4.1]{G-Found}). 

Assume now that the kernel $K$ of $T$ satisfies \eqref{EqEquivalenceKernel}. We have
\begin{align*}
|\langle Tf, g \rangle| = |\langle K, g \otimes \overline{f} \rangle| &\leq \norm{K}_{M^{\infty}_{\lambda \otimes (-\mu)}(\mathbb{R}^{d_2+d_1})} \norm{g \otimes \overline{f}}_{M^1_{(-\lambda) \otimes \mu}(\mathbb{R}^{d_2+d_1})} \\
&\leq \norm{K}_{M^{\infty}_{\lambda \otimes (-\mu)}(\mathbb{R}^{d_2+d_1})} \norm{g}_{M^1_{-\lambda}(\mathbb{R}^{d_2})} \norm{f}_{M^1_{\mu}(\mathbb{R}^{d_1})},
\end{align*}
for all $g \in M^1_{-\lambda}(\mathbb{R}^{d_2})$. Then, $Tf \in M^{\infty}_{\lambda}(\mathbb{R}^{d_2})$ and
$$ \norm{Tf}_{M^{\infty}_{\lambda}(\mathbb{R}^{d_2})} \leq \norm{K}_{M^{\infty}_{\lambda \otimes (-\mu)}(\mathbb{R}^{d_2+d_1})} \norm{f}_{M^1_{\mu}(\mathbb{R}^{d_1})} $$
for all $f \in \mathcal{S}_{\omega}(\mathbb{R}^{d_1})$. Hence, $T$ extends to a continuous map from $M^1_{\mu}(\mathbb{R}^{d_1})$ to $M^{\infty}_{\lambda}(\mathbb{R}^{d_2})$, and
$$ \norm{T}_{\op} \leq \norm{K}_{M^{\infty}_{\lambda \otimes (-\mu)}(\mathbb{R}^{d_2+d_1})}. $$
The uniqueness of the extension follows from the density of $\mathcal{S}_{\omega}(\mathbb{R}^{d_1})$ in $M^1_{\lambda}(\mathbb{R}^{d_1})$.
\end{proof}

When the operator $T$ considered in Proposition~\ref{PropConverse} is the Weyl operator $a^w(x,D)$, the characterization given in terms of the kernel can be transferred to the symbol $a$. In fact, the kernel $K$ of the Weyl operator is given in terms of the symbol $a$ by 
(see \eqref{EqWeylOperator}; cf.~\cite{Sh}) 
\begin{equation}\label{EqShubin}
K(x,y) = \mathcal{F}^{-1}_{\xi \mapsto x-y}(a)\Big(\frac{x+y}{2}, \xi \Big).
\end{equation}
From~\eqref{EqChangeT} we observe that $K = \Wig^{-1}[a]$, where $\Wig[\cdot]$ is the Wigner-like transform defined in \eqref{Wigner-like}. Hence, by Propositions~\ref{LemmaKEY}, \ref{PropClosedGraph} and~\ref{PropConverse}:
\begin{theo}\label{TheoCharWeyl}
Let $\omega$ be a subadditive weight function and $a \in \mathcal{S}'_{\omega}(\mathbb{R}^{2d})$.
\begin{itemize}
\item[(i)] The Weyl operator acts $a^w(x,D):\mathcal{S}_{\omega}(\mathbb{R}^d) \to \mathcal{S}_{\omega}(\mathbb{R}^d)$ if and only if for every $\lambda>0$ there exists $\mu>0:$ $\Wig^{-1}[a] \in M^{\infty}_{\lambda \otimes (-\mu)}(\mathbb{R}^{2d})$.
\item[(ii)] The Weyl operator $a^w(x,D):\mathcal{S}_{\omega}(\mathbb{R}^d) \to \mathcal{S}_{\omega}(\mathbb{R}^d)$ is compact if and only if there exists $\mu>0$ such that $\Wig^{-1}[a] \in \bigcap_{\lambda>0} M^{\infty}_{\lambda \otimes (-\mu)}(\mathbb{R}^{2d})$.
\end{itemize}
\end{theo}
\begin{proof}
$(i)$ If $a^w(x,D):\mathcal{S}_{\omega}(\mathbb{R}^d) \to \mathcal{S}_{\omega}(\mathbb{R}^d)$ is well defined, by Proposition~\ref{PropClosedGraph} it is continuous, and then the result follows by Propositions~\ref{LemmaKEY} and \ref{PropConverse}.

For the converse, given any $f \in \mathcal{S}_{\omega}(\mathbb{R}^d)$, let us see that
$ a^w(x,D)f \in \mathcal{S}_{\omega}(\mathbb{R}^d). $
By Proposition~\ref{PropConverse}, for every $\lambda>0$ there exists $\mu>0$ such that
$ a^w(x,D): M^1_{\mu}(\mathbb{R}^d) \to M^{\infty}_{\lambda}(\mathbb{R}^d) $
continuously. Since $f \in \mathcal{S}_{\omega}(\mathbb{R}^d) \subseteq M^1_{\mu}(\mathbb{R}^d)$, we have $a^w(x,D)f \in M^{\infty}_{\lambda}(\mathbb{R}^d)$, for every $\lambda>0$, and so
$$ a^w(x,D)f \in \bigcap_{\lambda>0} M^{\infty}_{\lambda}(\mathbb{R}^d) = \mathcal{S}_{\omega}(\mathbb{R}^d). $$
$(ii)$ If $a^w(x,D):\mathcal{S}_{\omega}(\mathbb{R}^d) \to \mathcal{S}_{\omega}(\mathbb{R}^d)$ is compact, then by Propositions~\ref{LemmaKEY} and~\ref{PropConverse}, we obtain the result.

Conversely, if for every $\lambda>0$, $\Wig^{-1}[a] \in M^{\infty}_{\lambda \otimes (-\mu)}(\mathbb{R}^{2d})$, then by $(i)$  the Weyl operator is continuous. Again an application of Propositions~\ref{LemmaKEY} and~\ref{PropConverse} gives that $a^w(x,D)$ is compact.
\end{proof}

\begin{exam}\label{ExamContCompWeylWig}
Observe first that the Wigner transform can be defined for $f,g \in \mathcal{S}'_{\omega}(\mathbb{R}^d)$ as
$$
\Wig(g,f)=\mathcal{F}_2 \mathcal{T}(g\otimes\overline{f}),
$$
where $\mathcal{F}_2$ is the second partial Fourier transform, $\mathcal{T}$ is the (unitary) change of variables \eqref{EqChangeT}, and both complex conjugate and tensor product are defined on distributions. Then, a corresponding Moyal formula holds in the frame of ultradistributions; indeed, for $h,k\in \mathcal{S}_{\omega}(\mathbb{R}^d)$ it is known that $\Wig (k,h)\in \mathcal{S}_{\omega}(\mathbb{R}^{2d})$, and we have
\begin{equation}\label{Moyal-distrib}
\langle \Wig(g,f), \Wig(k,h) \rangle = \langle\mathcal{F}_2 \mathcal{T}(g\otimes\overline{f}),\mathcal{F}_2 \mathcal{T}(k\otimes\overline{h})\rangle = \langle g\otimes\overline{f},k\otimes\overline{h}\rangle = \langle g,k \rangle \overline{\langle f,h\rangle},
\end{equation}
where the notation $\langle \cdot,\cdot\rangle$ indicates here the (conjugate-linear) duality product $\mathcal{S}'_{\omega}$-$\mathcal{S}_{\omega}$. Now let $a=\Wig(g,f)$ for some $f,g \in \mathcal{S}'_{\omega}(\mathbb{R}^d)$.  For every $h,k \in \mathcal{S}_{\omega}(\mathbb{R}^d)$, we then have by \eqref{Moyal-distrib}:
$$ \langle a^w(x,D)h, k \rangle = (2\pi)^{-d} \langle \Wig(g,f), \Wig(k,h) \rangle = (2\pi)^{-d} \langle g,k \rangle \overline{\langle f,h\rangle} = \langle (2\pi)^{-d} \overline{\langle f,h \rangle}g, k \rangle $$
that means that
$$ a^w(x,D)h = (2\pi)^{-d} \overline{\langle f,h \rangle} g. $$
Then, $a^w(x,D):\mathcal{S}_{\omega}(\mathbb{R}^d) \to \mathcal{S}_{\omega}(\mathbb{R}^d)$ if and only if $g \in \mathcal{S}_{\omega}(\mathbb{R}^d)$, for every $f\in\mathcal{S}'_{\omega}(\mathbb{R}^d)$. Since this operator has rank one, it is compact.

\end{exam}

By Example~\ref{ExamContCompWeylWig}, the function $a\in O_M^{\omega}(\mathbb{R}^{2d})$ given in Example~\ref{ExamMultCont}$(b)$ cannot be represented as $\Wig(g,f)$ for $g \in \mathcal{S}_{\omega}(\mathbb{R}^d)$ and $f \in \mathcal{S}'_{\omega}(\mathbb{R}^d)$. On the other hand, if $b(x,\xi) = e^{2ix\cdot \xi} f(\xi)$ with $f \in \mathcal{S}_{\omega}(\mathbb{R}^d)$, then $b^w(x,D):\,\mathcal{S}_{\omega}(\mathbb{R}^d)\to\mathcal{S}_{\omega}(\mathbb{R}^d)$ is well defined, continuous, and compact. In fact,
$ b = \Wig(\psi,1) $
for some suitable $\psi \in \mathcal{S}_{\omega}(\mathbb{R}^d)$, since
$$ \Wig(\psi,1)(x,\xi) = \int \psi(x+y/2) e^{-iy\cdot \xi} dy = 2^d \int \psi(s) e^{-2i(s-x)\cdot \xi} ds = 2^d e^{2ix\cdot \xi} \widehat{\psi}(2\xi). $$
Similarly, if we consider $c(x,\xi) = e^{-2ix\cdot \xi} g(x)$, for some $g \in \mathcal{S}_{\omega}(\mathbb{R}^d)$, then we have that
$ c = \Wig(\gamma, \delta) $
for some suitable $\gamma \in \mathcal{S}_{\omega}(\mathbb{R}^d)$. Indeed, from~\cite[Lemma 4.3.1]{G-Found} (proved with the duality product $\mathcal{S}'_{\omega}$-$\mathcal{S}_{\omega}$) it is easy to see that
$$ \Wig(\delta,\gamma)(x,\xi) = 2^d e^{2ix\cdot \xi} V_{\mathcal{I}\gamma}\delta(2x,2\xi) = 2^d e^{2ix\cdot \xi} \overline{\mathcal{I}\gamma(-2x)} = 2^d e^{2ix\cdot \xi} \overline{\gamma(2x)}, $$
and so
$$ \Wig(\gamma,\delta)(x,\xi) = \overline{\Wig(\delta,\gamma)(x,\xi)} = 2^d e^{-2ix\cdot \xi} \gamma(2x). $$

Now, we characterize the compactness of the Weyl operator via the STFT of the symbol. We need some preparation. Fixed $\psi \in \mathcal{S}_{\omega}(\mathbb{R}^{2d})\setminus\{0\}$, we set $\Psi := \mathcal{T}^{-1}\mathcal{F}_2^{-1}(\psi) \in \mathcal{S}_{\omega}(\mathbb{R}^{2d})$. We  check that (compare it with the formula in the proof of~\cite[Prop. 4.8]{ToftCont})
\begin{equation}\label{EqToftVicente}
V_{\Psi}K(x-y/2, x+y/2, \xi+\eta/2, -\xi+\eta/2) = (2\pi)^{-d} e^{iy \cdot \xi} V_{\psi}a(x,\xi,\eta,y).
\end{equation}
To show~\eqref{EqToftVicente}, we see what happens when commuting the phase-shift operator and $\mathcal{T}^{-1}$, where the operator $\mathcal{T}$ is given in \eqref{EqChangeT}:
\begin{lema}\label{LemmaShiftT-1}
For every $\psi \in \mathcal{S}_{\omega}(\mathbb{R}^{2d})\setminus\{0\}$, we have
$$ \Pi(x-y/2, x+y/2, \xi+\eta/2, -\xi+\eta/2) \mathcal{T}^{-1} \psi = \mathcal{T}^{-1}\Pi(x,-y,\eta,\xi)\psi, \qquad x,y,\eta,\xi \in \mathbb{R}^d. $$
\end{lema}
\begin{proof}
For $u,v \in \mathbb{R}^d$,
\begin{align*}
&\Pi(x-y/2, x+y/2, \xi+\eta/2, -\xi+\eta/2) \mathcal{T}^{-1} \psi(u,v) \\
& \qquad = e^{i((\xi+\eta/2)\cdot u + (-\xi+\eta/2)\cdot v)} \mathcal{T}^{-1}\psi(u-x+y/2, v-x-y/2) \\
& \qquad = e^{i\big(\frac{u+v}{2}\cdot \eta + (u-v)\cdot \xi\big)} \psi\Big(\frac{u+v}{2}-x, (u-v)-(-y)\Big) \\
& \qquad = \Pi(x,-y,\eta,\xi)\psi\Big(\frac{u+v}{2}, u-v\Big),
\end{align*}
which gives the result.
\end{proof}
We also use the next formula that relates the behaviour of the short-time Fourier transform of the second partial Fourier transform of $a\in\mathcal{S}'_{\omega}(\mathbb{R}^{2d})$ to that of $a$ (compare it to~\cite[(3.6)]{G-Found}).
\begin{lema}\label{LemmaSTFTPartial}
Let $a \in \mathcal{S}'_{\omega}(\mathbb{R}^{2d})$ and $\psi \in \mathcal{S}_{\omega}(\mathbb{R}^{2d}) \setminus \{0\}$. We have
$$ V_{\psi}a(x,-y,\eta,\xi) = (2\pi)^{-d} e^{iy \cdot \xi} V_{\mathcal{F}_2(\psi)}\mathcal{F}_2(a)(x,\xi,\eta,y), \qquad x,\xi,\eta,y \in \mathbb{R}^d. $$
\end{lema}
\begin{proof}
It is easy to see that, for $x_1,x_2,y_1,y_2 \in \mathbb{R}^d$:
\begin{align*}
T_{(x_1,x_2)}\mathcal{F}_2(\psi) &= \mathcal{F}_2(M_{(0,x_2)} T_{(x_1,0)}\psi), \\ 
M_{(y_1,y_2)}\mathcal{F}_2(\psi) &= \mathcal{F}_2(M_{(y_1,0)} T_{(0,-y_2)}\psi), \\ 
\Pi(x_1,x_2,y_1,y_2)\mathcal{F}_2(\psi) &= e^{ix_2 \cdot y_2} \mathcal{F}_2(\Pi(x_1,-y_2,y_1,x_2)\psi).
\end{align*}
Then, by the definition of the STFT,
\begin{align*}
V_{\mathcal{F}_2(\psi)}\mathcal{F}_2(a)(x,\xi,\eta,y) &= \langle \mathcal{F}_2(a), \Pi(x,\xi,\eta,y) \mathcal{F}_2(\psi) \rangle \\
&= e^{-iy\cdot \xi} \langle \mathcal{F}_2(a), \mathcal{F}_2(\Pi(x,-y,\eta,\xi)\psi) \rangle \\
&= (2\pi)^d e^{-iy\cdot \xi} \langle a, \Pi(x,-y,\eta,\xi)\psi \rangle,
\end{align*}
from which we obtain the result.
\end{proof}
For $\Psi := \mathcal{T}^{-1}\mathcal{F}_2^{-1}(\psi)$, by \eqref{EqShubin} and
Lemma~\ref{LemmaShiftT-1} we have that
\begin{align*}
& V_{\Psi}K(x-y/2,x+y/2,\xi+\eta/2,-\xi+\eta/2) \\
& \qquad = \langle K, \Pi(x-y/2,x+y/2,\xi+\eta/2,-\xi+\eta/2)\mathcal{T}^{-1}\mathcal{F}_2^{-1}(\psi) \rangle \\
& \qquad = \langle \mathcal{T}^{-1}\mathcal{F}_2^{-1}(a), \mathcal{T}^{-1}\Pi(x,-y,\eta,\xi)\mathcal{F}^{-1}_2(\psi) \rangle \\
& \qquad = \langle \mathcal{F}_2^{-1}(a), \Pi(x,-y,\eta,\xi)\mathcal{F}^{-1}_2(\psi) \rangle \\
& \qquad = V_{\mathcal{F}^{-1}_2(\psi)}\mathcal{F}_2^{-1}(a)(x,-y,\eta,\xi).
\end{align*}
Hence, by Lemma~\ref{LemmaSTFTPartial} (replacing $\psi$ and $a$ by $\mathcal{F}_2^{-1}(\psi)$ and $\mathcal{F}_2^{-1}(a)$), we see that the equality in~\eqref{EqToftVicente} holds and therefore by Propositions~\ref{PropConverse} and~\ref{LemmaKEY}  (since $V_{\psi}a$ of any $a\in\mathcal{S}'_\omega(\R^d)$ is continuous~\cite[Theorem 2.4]{GZ}, we can replace the essential supremum by the supremum in the result below):
\begin{theo}\label{TheoContCompWeyl}
Let $\omega$ be a subadditive weight function, $0 \neq \psi \in \mathcal{S}_{\omega}(\mathbb{R}^{2d})$, and $a \in \mathcal{S}'_{\omega}(\mathbb{R}^{2d})$.
\begin{itemize}
\item[(i)] The Weyl operator acts $a^w(x,D):\mathcal{S}_{\omega}(\mathbb{R}^d)\to \mathcal{S}_{\omega}(\mathbb{R}^d)$ if and only if for every $\lambda>0$, there exists $\mu>0$ such that
\begin{equation}\label{EqCharWeyl}
\sup_{x,\xi,\eta,y \in \mathbb{R}^d} |V_{\psi}a(x,\xi,\eta,y)| e^{\lambda \omega(x-y/2, \xi+\eta/2)} e^{-\mu\omega(x+y/2, -\xi+\eta/2)} < +\infty.
\end{equation}
\item[(ii)] The Weyl operator $a^w(x,D):\mathcal{S}_{\omega}(\mathbb{R}^d) \to \mathcal{S}_{\omega}(\mathbb{R}^d)$ is compact if and only if there exists $\mu>0$ such that for every $\lambda>0$,~\eqref{EqCharWeyl} is satisfied.
\end{itemize}
\end{theo}

\begin{cor}\label{CorCompactConvol}
Let $\omega$ be a subadditive weight function and let $a \in \mathcal{S}'_{\omega}(\mathbb{R}^{2d})$. If the Weyl operator $a^w(x,D):\mathcal{S}_{\omega}(\mathbb{R}^d) \to \mathcal{S}_{\omega}(\mathbb{R}^d)$ is compact, then $a \in (O_C^{\omega})'(\mathbb{R}^{2d})$.
\end{cor}
\begin{proof}
Fix $0 \neq \psi \in \mathcal{S}_{\omega}(\mathbb{R}^{2d})$. By Theorem~\ref{TheoContCompWeyl}$(ii)$, there exists $\mu>0$ such that for every $\lambda>0$ there exists $C_\lambda>0$ such that 
$$
|V_{\psi}a(x,\xi,\eta,y)|\le C_\lambda e^{-(\lambda+\mu) \omega(x-y/2, \xi+\eta/2)} e^{\mu\omega(x+y/2, -\xi+\eta/2)}, \ \ x,\xi,\eta,y \in \mathbb{R}^d.$$ 
Now, since
$$ -(\lambda+\mu) \omega(x-y/2,\xi+\eta/2)) \leq -(\lambda+\mu)\omega(x,\xi) + (\lambda+\mu)\omega(-y/2,\eta/2) $$
and $\mu\omega(x+y/2, -\xi+\eta/2) \leq \mu\omega(x,\xi) + \mu\omega(y/2,\eta/2)$, we obtain
$$  |V_{\psi}a(x,\xi,\eta,y)|\le C_\lambda e^{-\lambda\omega(x,\xi)} e^{(\lambda+2\mu)\omega(y,\eta)},\ \  x,\xi,\eta,y \in \mathbb{R}^d. $$
The result then follows by Theorem~\ref{TheoCharConvSTFT}.
\end{proof}
We note that the converse of Corollary~\ref{CorCompactConvol} does not hold. In fact, the function in Example~\ref{ExamMultCont}$(b)$ is an $\omega$-convolutor. For $\psi \in \mathcal{S}_{\omega}(\mathbb{R}^{2d})$,
\begin{align*}
V_{\psi}a(x,\xi,\eta,y) &= \int e^{-it\eta} e^{-isy} e^{-2its} f(s) \overline{\psi(t-x,s-\xi)} dt ds \\
&= \int e^{-i(t+x)\eta} e^{-isy} e^{-2i(t+x)s} f(s) \overline{\psi(t,s-\xi)} dt ds \\
&= e^{-ix\eta} \int e^{-is(y+2x)} e^{-it(\eta+2s)} f(s) \overline{\psi(t,s-\xi)} dt ds \\
&= e^{-ix\eta} \int e^{-is(y+2x)} f(s) \mathcal{F}_1\big(\overline{\psi}\big)(\eta+2s, s-\xi) ds.
\end{align*}
Since $f \in \mathcal{S}_{\omega}(\mathbb{R}^d)$ and $\psi \in \mathcal{S}_{\omega}(\mathbb{R}^{2d})$, we have that
$$ f(X) \mathcal{F}_1\big(\overline{\psi}\big)(Y,Z) \in \mathcal{S}_{\omega}(\mathbb{R}^{3d}_{(X,Y,Z)}). $$
Since the change of variables
$ X=s, Y=\eta+2s, Z=s-\xi $
is invertible,  we have
$$ f(s) \mathcal{F}_1\big(\overline{\psi}\big)(\eta+2s, s-\xi) \in \mathcal{S}_{\omega}(\mathbb{R}^{3d}_{(s,\eta,\xi)}). $$
Therefore,
$$ \mathcal{F}_{s\mapsto w} \Big( f(s) \mathcal{F}_1 \big(\overline{\psi}\big)(\eta+2s, s-\xi) \Big) \in \mathcal{S}_{\omega}(\mathbb{R}^{3d}_{(w,\eta,\xi)}),  $$
and (with $w=y+2x$) for every $\lambda>0$ there exists $C_{\lambda}>0$ such that
$$ |V_{\psi}a(x,\xi,\eta,y)| \leq C_{\lambda} e^{-\lambda \omega(y+2x)} e^{-\lambda \omega(\xi)} e^{-\lambda \omega(\eta)} \leq C_{\lambda} e^{-\lambda\omega(x,\xi)} e^{\lambda\omega(\eta,y)}. $$

\section{Applications and further results}\label{SectLoc}
In this section we apply our previous study to the compactness of the localization operator when acting from  $\mathcal{S}_{\omega}(\mathbb{R}^d)$ into itself, and we also provide some examples and remarks. First, we observe that if $a \in \mathcal{S}'_{\omega}(\mathbb{R}^{2d})$, then $L^a_{\psi,\gamma}:\mathcal{S}_{\omega}(\mathbb{R}^d) \to \mathcal{S}'_{\omega}(\mathbb{R}^d)$ is linear and continuous, so the localization operator $L^a_{\psi,\gamma}:\mathcal{S}_{\omega}(\mathbb{R}^d)\to\mathcal{S}_{\omega}(\mathbb{R}^d)$ is well defined if and only if it is continuous (Lemma~\ref{well-def}). By~\eqref{EqRelLocWeyl} and Theorems~\ref{TheoCharWeyl} and~\ref{TheoContCompWeyl}, we obtain immediately the following characterization:
\begin{theo}\label{TheoContCompLocTrivial}
Let $\omega$ be a subadditive weight function. Let $\psi,\gamma,\varphi \in \mathcal{S}_{\omega}(\mathbb{R}^d)$ be non-zero windows and $a \in \mathcal{S}'_{\omega}(\mathbb{R}^{2d})$.
\begin{itemize}
\item[(i)] The localization operator $L^a_{\psi,\gamma}:\mathcal{S}_{\omega}(\mathbb{R}^d) \to \mathcal{S}_{\omega}(\mathbb{R}^d)$ is well defined (and continuous) if and only if for every $\lambda>0$ there exists $\mu>0$ such that $\Wig^{-1}[a \ast \Wig(\gamma, \psi)] \in M^{\infty}_{\lambda \otimes (-\mu)}(\mathbb{R}^{2d})$ if and only if for every $\lambda>0$, there exists $\mu>0$ such that
\begin{equation}\label{EqCharLoc}
\esup_{x,\xi,\eta,y \in \mathbb{R}^d} |V_{\varphi}(a \ast \Wig(\gamma, \psi))(x,\xi,\eta,y)| e^{\lambda \omega(x-y/2, \xi+\eta/2)} e^{-\mu\omega(x+y/2, -\xi+\eta/2)} < +\infty.
\end{equation}
\item[(ii)] The localization operator $L^a_{\psi,\gamma}:\mathcal{S}_{\omega}(\mathbb{R}^d) \to \mathcal{S}_{\omega}(\mathbb{R}^d)$  is compact if and only if there exists $\mu>0$ such that $\Wig^{-1}[a \ast \Wig(\gamma, \psi)]\in\bigcap_{\lambda>0} M^{\infty}_{\lambda \otimes (-\mu)}(\mathbb{R}^{2d})$ if and only if there exists $\mu>0$ such that for every $\lambda>0$,~\eqref{EqCharLoc} is satisfied.
\end{itemize}
\end{theo}


Then we have:
\begin{theo}\label{compactness-implications}
Let $\omega$ be a subadditive weight function and $a \in \mathcal{S}'_{\omega}(\mathbb{R}^{2d})$. Consider the following conditions:
\begin{itemize}
\item[(i)] The Weyl operator $a^w(x,D):\mathcal{S}_{\omega}(\mathbb{R}^d)\to\mathcal{S}_{\omega}(\mathbb{R}^d)$ is compact.
\item[(ii)] $a \in (O_C^{\omega})'(\mathbb{R}^{2d})$.
\item[(iii)] The localization operator $L^a_{\psi,\gamma}:\mathcal{S}_{\omega}(\mathbb{R}^d) \to \mathcal{S}_{\omega}(\mathbb{R}^d)$ is compact for every  $\psi,\gamma\in \mathcal{S}_{\omega}(\mathbb{R}^d)$.
\end{itemize}
We have $(i)\Rightarrow (ii) \Rightarrow (iii)$. 
\end{theo}

\begin{proof}
$(i)\Rightarrow(ii)$ follows from Corollary~\ref{CorCompactConvol}.

$(ii)\Rightarrow(i)$: If $a \in (O_C^{\omega})'(\mathbb{R}^{2d})$, then $a\ast \Wig(\gamma,\psi) \in \mathcal{S}_{\omega}(\mathbb{R}^{2d})$ for any $\psi,\gamma \in \mathcal{S}_{\omega}(\mathbb{R}^d)$ and hence $L^a_{\psi,\gamma} = (a\ast \Wig(\gamma,\psi))^w(x,D):\mathcal{S}_{\omega}(\mathbb{R}^d)\to \mathcal{S}_{\omega}(\mathbb{R}^d)$ is compact for every  $\psi,\gamma \in \mathcal{S}_{\omega}(\mathbb{R}^d)$ by Theorem~\ref{TheoContCompLocTrivial}.
\end{proof}

\begin{exam}

We have seen in Section~\ref{Weyl oper} that $b(x,\xi) = e^{2ix\cdot \xi} f(\xi)$, being $f\in\mathcal{S}_{\omega}(\mathbb{R}^d)$, coincides with $\Wig(\psi,1)$ for some $\psi \in \mathcal{S}_{\omega}(\mathbb{R}^{2d})$. Therefore, using~\cite[Lemma 14.5.1(b)]{G-Found}
(proved with the duality product $ \mathcal{S}'_{\omega}$-$ \mathcal{S}_{\omega}$), we have, for any $\gamma \in \mathcal{S}_{\omega}(\mathbb{R}^{2d})$,
\begin{align*}
V_{\Wig \gamma}(\Wig (\psi,1))(x,y,\xi,\eta) &= (2\pi)^{d/2} e^{-iy\cdot\eta} \overline{V_{\gamma}\psi(x+\eta/2, y-\xi/2)} V_{\gamma}1(x-\eta/2, y+\xi/2) \\
&= (2\pi)^{d/2} e^{-iy \cdot \eta} \overline{V_{\gamma}\psi(x+\eta/2, y-\xi/2)} e^{-i(x-\eta/2)\cdot (y+\xi/2)} \overline{\widehat{\gamma}(-y-\xi/2)}.
\end{align*}
Since $\psi \in \mathcal{S}_{\omega}(\mathbb{R}^{2d})$, we have that for all $\lambda>0$ there exists $C>0$ such that
$$ |V_{\Wig \gamma}(\Wig (\psi,1))(x,y,\xi,\eta)| \leq C e^{-\lambda\omega(x+\eta/2, y-\xi/2)} \leq C e^{-\lambda\omega(x,y)} e^{\lambda\omega(\xi,\eta)}. $$
Hence, by Theorem~\ref{TheoCharConvSTFT}, we have  $b \in (O_C^{\omega})'(\mathbb{R}^{2d})$. We can proceed similarly for $c(x,\xi) = e^{-2ix\cdot \xi}g(x)$, for some $g\in\mathcal{S}_{\omega}(\mathbb{R}^d)$, and show that $c \in (O_C^{\omega})'(\mathbb{R}^{2d})$. Hence, $L^b_{\psi,\gamma}$ and $L^c_{\psi,\gamma}$ are compact when acting from $\mathcal{S}_{\omega}(\mathbb{R}^d)$ into itself for every  $\psi,\gamma \in \mathcal{S}_{\omega}(\mathbb{R}^d)$, by Theorem~\ref{compactness-implications}.
\end{exam}
\begin{nota}
\label{nota64}
If $a(x,\xi)=e^{-2ix\cdot\xi}f(\xi)$ for some $f\in
\mathcal S_\omega(\R^d)$ is the symbol considered in Example~\ref{ExamMultCont}$(b)$, we see that 
$a$ is an $\omega$-convolutor. Indeed, for $\psi_1, \psi_2 \in \mathcal{S}_{\omega}(\mathbb{R}^d)$, we consider $\psi:= \psi_1 \otimes \psi_2 \in \mathcal{S}_{\omega}(\mathbb{R}^{2d})$. We have for all $z=(z_1,z_2), \eta=(\eta_1,\eta_2)\in\mathbb{R}^{2d}$ that
\begin{align*}
V_{\psi}a(z,\eta) &= \int_{\mathbb{R}^{2d}} a(x,\xi) \overline{\psi(x-z_1,\xi-z_2)} e^{-i(x,\xi)\cdot(\eta_1,\eta_2)} dx d\xi \\
&= \int_{\mathbb{R}^{2d}} e^{-2ix\cdot \xi} f(\xi) \overline{\psi_1(x-z_1)} \overline{\psi_2(\xi-z_2)} e^{-i(x\cdot \eta_1 + \xi\cdot \eta_2)} dx d\xi \\
&= \int_{\mathbb{R}^{d}} e^{-ix\cdot \eta_1} \overline{\psi_1(x-z_1)} \Big( \int_{\mathbb{R}^{d}} f(\xi) \overline{\psi_2(\xi-z_2)} e^{-i\xi\cdot (2x+\eta_2)} d\xi \Big) dx \\
&= \int_{\mathbb{R}^{d}} e^{-ix\cdot \eta_1} \overline{\psi_1(x-z_1)} V_{\psi_2}f(z_2, 2x+\eta_2) dx.
\end{align*}
Now assume that there exist $c \in \mathbb{R}$ and $b>0$ such that $\omega(t) \geq c+b\log(1+t)$. Since $\psi_1 \in \mathcal{S}_{\omega}(\mathbb{R}^d)$, for every $\mu>0$ there exists $C_{\mu}>0$ such that
$$ |\overline{\psi_1(x-z_1)}| \leq C_{\mu} e^{-\mu\omega(x-z_1)} \leq C_{\mu} e^{-\mu\omega(z_1)} e^{\mu\omega(x)} $$
and as $f \in \mathcal{S}_{\omega}(\mathbb{R}^d)$ there exists $C'_{\mu}>0$ such that
\begin{align*}
|V_{\psi_2}f(z_2,2x+\eta_2)| \leq C'_{\mu} e^{-2(\mu+(d+1)/b)\omega(z_2,2x+\eta_2)} &\leq C'_{\mu} e^{-\mu\omega(z_2)} e^{-(\mu+(d+1)/b)\omega(2x+\eta_2)} \\
&\leq C'_{\mu} e^{-\mu\omega(z_2)} e^{-(\mu+(d+1)/b)\omega(x)} e^{(\mu+(d+1)/b)\omega(\eta_2)}.
\end{align*}
Thus, we obtain that for all $\mu>0$ there exists $\lambda_{\mu}=\mu+(d+1)/b>0$ such that
$$ |V_{\psi}a(z,\eta)| \leq C_{\mu} C'_{\mu} e^{-\mu(\omega(z_1)+\omega(z_2))} e^{\lambda_{\mu}\omega(\eta_2)} \int_{\mathbb{R}^d} e^{-\frac{d+1}{b}\omega(x)} dx. $$
The integral converges and hence, by Theorem~\ref{TheoCharConvSTFT}, we obtain that $a \in (O_C^{\omega})'(\mathbb{R}^{2d})$.

It follows, by Theorem~\ref{compactness-implications}, that
the localization operator $L^a_{\psi,\gamma}:\mathcal{S}_{\omega}(\mathbb{R}^d)\to \mathcal{S}_{\omega}(\mathbb{R}^d)$ is compact for every  $\psi,\gamma \in \mathcal{S}_{\omega}(\mathbb{R}^d)$.
 However, we have also seen in Example~\ref{ExamMultCont}$(b)$ that the corresponding Weyl operator $a^w(x,D)$, defined on $\mathcal{S}_{\omega}(\mathbb{R}^d)$, does not take values in 
 $\mathcal{S}_{\omega}(\mathbb{R}^d)$. So (ii) cannot imply (i) in Theorem~\ref{compactness-implications}. 
	
	On the other hand, that (iii) implies (ii) is also far to be true. In fact, we first observe that there is the following relation between the Wigner-like transform and the cross-Wigner transform: given two functions $f,g\in L^2(\R^d)$, one has
	$$
	\Wig[f\otimes \overline g]=\Wig(f,g),
	$$
where we denote $(f\otimes \overline{g})(x,y):=f(x)\cdot \overline{g(y)}$, for $x,y\in\R^d.$	Now, we consider  the projective tensor product $\mathcal{S}_{\omega}(\mathbb{R}^d)\otimes_\pi \mathcal{S}_{\omega}(\mathbb{R}^d)$ that, as $\mathcal{S}_{\omega}(\mathbb{R}^d)$ is metrizable and barrelled, it is also barrelled \cite[15.6.6]{Jarchow}. Moreover, it is dense in $\mathcal{S}_{\omega}(\mathbb{R}^{2d})$. Since the Wigner-like transform is an isomorphism 
from $\mathcal{S}_{\omega}(\mathbb{R}^{2d})$ to $\mathcal{S}_{\omega}(\mathbb{R}^{2d})$ and  $$A:=\text{span} \big\{\Wig(\varphi,\psi)\,:\,\varphi,\psi\in \mathcal{S}_{\omega}(\mathbb{R}^d)\big\}$$ coincides with $\Wig[\mathcal{S}_{\omega}(\mathbb{R}^d)\otimes_\pi \mathcal{S}_{\omega}(\mathbb{R}^d)]$, we deduce that $A$ is a dense and barrelled subspace of $\mathcal{S}_{\omega}(\mathbb{R}^{2d})$. Being $\mathcal{S}_{\omega}(\mathbb{R}^{d})$ a nuclear space \cite{BJOS-2021,BJOS-2020}, the induced topology in $A$ from $\mathcal{S}_{\omega}(\mathbb{R}^{2d})$ is the same as the topology of $\mathcal{S}_{\omega}(\mathbb{R}^d)\otimes_\pi \mathcal{S}_{\omega}(\mathbb{R}^d)$. Hence, any ultradistribution $a\in \mathcal{S}'_{\omega}(\mathbb{R}^{2d})$ satisfies $a \in (O_C^{\omega})'(\mathbb{R}^{2d})$ whenever $a*\Wig(\varphi,\psi)\in \mathcal{S}_{\omega}(\mathbb{R}^{2d})$ for every $\psi,\varphi\in \mathcal{S}_{\omega}(\mathbb{R}^d)$. In fact, for $a\in \mathcal{S}'_{\omega}(\mathbb{R}^{2d})$, the convolution operator $\mathcal{C}_a:\mathcal{S}_{\omega}(\mathbb{R}^{2d})\to \mathcal{S}'_{\omega}(\mathbb{R}^{2d})$ is obviously continuous. Its restriction to $A$, $\mathcal{C}_a|_A:A\to \mathcal{S}_{\omega}(\mathbb{R}^{2d})$, has closed graph. Since $A$ is barrelled with the induced topology from  $\mathcal{S}_{\omega}(\mathbb{R}^{2d})$, $\mathcal{C}_a|_A$ is also continuous. Finally, as $A$ is dense in $\mathcal{S}_{\omega}(\mathbb{R}^{2d})$, its extension to $\mathcal{S}_{\omega}(\mathbb{R}^{2d})$ is continuous and coincides with $\mathcal{C}_a$, which shows that $a$ is an $\omega$-convolutor. In particular, we have that $\Wig^{-1}[a*\Wig(\varphi,\psi)]\in \mathcal{S}_{\omega}(\mathbb{R}^{2d})$ for every $\varphi,\psi\in \mathcal{S}_{\omega}(\mathbb{R}^d)$, which is, in view of Theorem~\ref{TheoContCompLocTrivial}, a stronger condition than $$\text{Wig}^{-1}[a*\Wig(\varphi,\psi)]\in \bigcap_{\lambda>0} M^{\infty}_{\lambda \otimes (-\mu)}(\mathbb{R}^{2d}),$$
for some $\mu>0$ and every $\varphi,\psi\in \mathcal{S}_{\omega}(\mathbb{R}^d)$. So, (iii) cannot imply (ii) in Theorem~\ref{compactness-implications}. Indeed, the following example gives a compact localization operator whose symbol is not an $\omega$-convolutor for some concrete windows.
\end{nota}

\begin{exam} We recall the following notation from \cite{G-Found}, but adapted to our definition of Fourier transform. Given two functions $f,g\in L^2(\R^d)$, we denote the \emph{cross-ambiguity function} of $f$ and $g$ by 
 \begin{equation}
 	\label{cross-amb}
 	A(f,g)(x,\xi)=\int_{\R^d}f\big(t+\frac{x}{2}\big)\overline{g\big(t-\frac{x}{2}\big)}e^{-i t\cdot \xi}dt=e^{\frac{i}{2}x\cdot\xi}V_gf(x,\xi).
 \end{equation}
We also denote by $\mathcal{U}$ the \emph{rotation operator} $\mathcal{U}F(x,\xi)=F(\xi,-x)$ of a function $F$ in $\R^{2d}$, and by $\mathcal{I}$ the reflection operator introduced in Section~\ref{preli}. Assume now that $f,g \in \mathcal{S}_{\omega}(\mathbb{R}^d)\setminus\{0\}$.  Using \cite[Lemma 4.3.4]{G-Found} and the properties of the Fourier transform, we have
\begin{equation}
\begin{aligned}\label{amb}
\mathcal{F}\Wig(f,g)(x,\xi) &= \mathcal{F}{\mathcal{F}\mathcal{U}A(f,g)}(x,\xi)= (2\pi)^{2d}\mathcal{I}\,\mathcal{U}A(f,g)(x,\xi)\\
&= (2\pi)^{2d}A(f,g)(-\xi,x)=(2\pi)^{2d} e^{-\frac{i}{2}x\cdot\xi}V_gf(-\xi,x)
\end{aligned}
	\end{equation}
On the other hand, by \cite[Lemma 4.3.1]{G-Found}, with our definition of Fourier transform, we have
\begin{equation*}
		\Wig(f,g)(x,\xi)=2^d e^{2ix\cdot \xi} V_{\mathcal{I}g}f(2x,2\xi).
\end{equation*}
Hence, we obtain 
\begin{equation*}
	\Wig(f,g)\Big(-\frac{\xi}{2},\frac{x}{2}\Big)=2^d e^{-\frac{i}{2}x\cdot \xi} V_{\mathcal{I}g}f(-\xi,x).
\end{equation*}
Consequently
\begin{equation}
	\label{wig-st}
V_{g}f(-\xi,x)=2^{-d}e^{\frac{i}{2}x\cdot \xi} \Wig(f,\mathcal{I}g)\Big(-\frac{\xi}{2},\frac{x}{2}\Big).
\end{equation}
So, combining \eqref{amb} and \eqref{wig-st}, we deduce
\begin{equation}
	\label{wig-for}
	\mathcal{F}\Wig(f,g)(x,\xi)=(\sqrt{2}\pi)^{2d}\Wig(f,\mathcal{I}g)\Big(-\frac{\xi}{2},\frac{x}{2}\Big).
\end{equation}
Suppose that $d=1$ and take now $f = \mathcal{I}g$ and assume $\supp(f) \subseteq [a,b]$ for some $a,b \in \mathbb{R}$ with $a<b$. As in \cite[Lemma 4.3.5]{G-Found}, we deduce that
$$ \mathcal{F}\Wig(f,\mathcal{I}f)(x,\xi) = (\sqrt{2}\pi)^{2d} \Wig(f,f)(-\xi/2, x/2) = 0, \ \text{for}\  \xi \notin [-2b,-2a]. $$

Now, we consider $a \in \mathcal{S}'_{\omega}(\mathbb{R}^{2})$ such that $\widehat{a}(x,\xi) = 1$ for every $\xi \in [-2b,-2a]$ and zero otherwise. In particular, $\widehat{a} \notin O_M^{\omega}(\mathbb{R}^{2d})$, and hence $a \notin (O_C^{\omega})'(\mathbb{R}^{2d})$. Then
$$ a \ast \Wig(f, \mathcal{I}f) = \mathcal{F}^{-1} (\widehat{a} \cdot \mathcal{F}\Wig(f,\mathcal{I}f)) = \Wig(f, \mathcal{I}f) \in \mathcal{S}_{\omega}(\mathbb{R}^{2}). $$
So  $L^a_{f, \mathcal{I}f}:\mathcal{S}_{\omega}(\mathbb{R})\to \mathcal{S}_{\omega}(\mathbb{R})$ is compact.
\end{exam}

Fern\'andez and Galbis showed~(\cite[Theorem 3.15]{FG-Compact} and \cite[Theorem 4.7]{FG-Some}) that, for $a \in M^{\infty}(\mathbb{R}^{2d})$, the localization operator $L^a_{\psi,\gamma}$ is compact in $L^2(\mathbb{R}^d)$ ($M^{p,q}(\mathbb{R}^d)$, $1 \leq p,q \leq \infty$) for every $\psi,\gamma\in \mathcal{S}(\mathbb{R}^d)$ if and only if, given $0\neq g_0 \in \mathcal{S}(\mathbb{R}^{2d})$, for every $R>0$, the following holds:
\begin{equation}\label{EqFG}
\lim_{|z|\to\infty} \sup_{|\eta| \leq R} |V_{g_0}a(z,\eta)| = 0.
\end{equation}
By Theorem~\ref{TheoCharConvSTFT}, if $a\in (O_C)'(\mathbb{R}^{2d})$, then  it satisfies~\eqref{EqFG}. Thus:
\begin{theo}\label{TheoSL2}
Let $1\leq p,q\leq +\infty$. If $a \in M^{\infty}(\mathbb{R}^{2d}) \cap (O_C)'(\mathbb{R}^{2d})$, then the localization operator $L^a_{\psi,\gamma}$ is compact for every  $\psi,\gamma\in \mathcal{S}(\mathbb{R}^d)$ when acting from $\mathcal{S}(\mathbb{R}^d)$ into itself and when acting from $M^{p,q}(\mathbb{R}^d)$ into itself.
\end{theo}


We can also generalize Theorem~\ref{TheoSL2} for every subadditive weight function $\omega$ when $1\leq p,q < +\infty$. With minor modifications in the proof of~\cite[Lemma 0.6]{BdM}, we obtain~\cite[Theorem 0.7]{BdM} under a weaker assumption (in fact, observe that the exponential in the limit of the next result does not depend on the variable $\eta$):
\begin{theo}\label{TheoChiaraComp}
Let $\omega$ be subadditive and $1 \leq p,q <+\infty$. Fix $\lambda\ge 0$. Given  $0\neq g_0 \in \mathcal{S}_{\omega}(\mathbb{R}^d)$, if $a \in M^{\infty}_{\lambda}(\mathbb{R}^{2d})$ satisfies, for every $R>0$,
$$ \lim_{|z|\to\infty} \sup_{|\eta| \leq R} |V_ga(z,\eta)| e^{\lambda \omega(z)} = 0, $$
then $L^{a}_{\psi,\gamma}: M^{p,q}_{\lambda}(\mathbb{R}^d) \to M^{p,q}_{\lambda}(\mathbb{R}^d)$ is compact for every $\psi,\gamma\in\mathcal{S}_{\omega}(\mathbb{R}^d)$.
\end{theo}
As a consequence, we obtain 
\begin{cor}
Let $\omega$ be subadditive and $1 \leq p,q <+\infty$. Fix $\lambda\ge 0$. If $a \in M^{\infty}_{\lambda}(\mathbb{R}^{2d}) \cap (O_C^{\omega})'(\mathbb{R}^{2d})$, then the localization operator $L^a_{\psi,\gamma}$ is compact for every $\psi, \gamma \in \mathcal{S}_{\omega}(\mathbb{R}^d)$ when acting from $\mathcal{S}_{\omega}(\mathbb{R}^d)$ into itself  and when acting from  $M^{p,q}_{\lambda}(\mathbb{R}^d)$ into itself.
\end{cor}

 {\bf Acknowledgments.} 
 The authors are very grateful to Jos\'e Bonet for helpful discussions regarding Remark~\ref{nota64}.
 Asensio is  supported by the project GV PROMETEU/2021/070.
Boiti and Oliaro have been partially supported by the INdAM - GNAMPA Projects 2020 ``Analisi microlocale e applicazioni: PDE's stocastiche e di evoluzione, analisi tempo-frequenza, variet\`{a}'' and 2023 ``Analisi di Fourier e Analisi Tempo-Frequenza
di Spazi Fun\-zio\-na\-li ed Ope\-ra\-to\-ri". Boiti is partially supported by the Projects FAR 2021, FAR 2022,
FIRD 2022 and FAR 2023 (University of Ferrara).
Jornet is partially supported by the project PID2020-\-119457GB-\-100 funded by MCIN/AEI/10.13039/501100011033 and by ``ERDF A way of making Europe''.

\end{document}